\documentclass[final]{siamltex}
\usepackage{float}
\usepackage{graphicx,subfigure}
\usepackage{multirow}
\usepackage{diagbox}
\usepackage{cancel}
\usepackage{amsmath}
\usepackage{amssymb}
\usepackage{amsfonts}
\usepackage{nicefrac}
\usepackage{color}
\newtheorem{thm}{Theorem}[section]

\newtheorem{lem}[thm]{Lemma}
\newtheorem{prop}[thm]{Proposition}
\newtheorem{rem}[thm]{Remark}
\numberwithin{equation}{section}

\usepackage{epstopdf}

\begin{document}
\title{A stabilized second order exponential time differencing multistep method for thin film growth model without slope selection}

\author{
Wenbin Chen\thanks{
Shanghai Key Laboratory 
of Mathematics for Nonlinear Sciences, School of Mathematical Sciences; Fudan University, Shanghai, China 200433 ({\tt wbchen@fudan.edu.cn})}
\and
Weijia Li\thanks{
School of Mathematical Sciences; Fudan University, Shanghai, China 200433 ({\tt 15110180013@fudan.edu.cn})}
\and
Zhiwen Luo\thanks{
School of Mathematical Sciences; Fudan University, Shanghai, China 200433 ({\tt 15210180015@fudan.edu.cn})}
\and
Cheng Wang\thanks{Mathematics Department; University of Massachusetts; North Dartmouth, MA 02747, USA ({\tt corresponding author: cwang1@umassd.edu})}
\and
Xiaoming Wang\thanks{Department of Mathematics, Southern University of Science and Technology, Shenzhen, China 518055, and Fudan University, Shanghai, China 200433, and Florida State University, Tallahassee, FL 32306, USA ({\tt wxm@math.fsu.edu})}
			}

\maketitle

	\begin{abstract}
In this paper, a stabilized second order in time accurate linear exponential time differencing (ETD) scheme for the no-slope-selection thin film growth model is presented. An artificial stabilizing term $A\tau^2\frac{\partial\Delta^2 u}{\partial t}$ is added to the physical model to achieve energy stability, with ETD-based multi-step approximations and Fourier collocation spectral method applied in the time integral and spatial discretization of the evolution equation, respectively. Long time energy stability and detailed $\ell^{\infty}(0,T; \ell^2)$ error analysis are provided based on the energy method, with a careful estimate of the aliasing error. In addition, numerical experiments are presented to demonstrate the energy decay and convergence rate.
	\end{abstract}



\pagestyle{myheadings}
\thispagestyle{plain}
\markboth{W.~CHEN, W.~LI, Z.~LUO, C.~WANG AND X.~WANG}{sETDMs2 for no-slope-selection thin film model}

\section{Introduction}
Consider the continuum model of no-slope-selection thin film epitaxial growth, which is the $L^2$ gradient flow of the following energy functional:
\begin{equation}\label{MBE energy}
E(u) = \int_\Omega \frac{\varepsilon^2}{2}|\Delta u|^2 -\frac{1}{2}\ln (1+|\nabla u|^2) ~\mbox{d}\mathbf{x},
\end{equation}
where $\Omega = [0,L]^2$, $u:\Omega\times [0,T]\rightarrow \mathbb{R}$ is a scaled height function of thin film with periodic boundary condition, and $\varepsilon>0$ is a constant. Here the first term in $E(u)$ represents the surface diffusion, and the logarithm term models the Ehrlich-Schowoebel (ES) effect which describes the effect of kinetic asymmetry in the adatom attachment-detachment; see \cite{ehrlich1966atomic,li2006high,li2003thin, schwoebel1969step}. Consequently, the growth equation is the $L^2$ gradient flow of \eqref{MBE energy}:
\begin{align}\label{MBE}
\frac{\partial u}{\partial t} = -\varepsilon^2\Delta^2 u - \nabla\cdot\left(\frac{\nabla u}{1+|\nabla u|^2}\right), ~x\in\Omega, ~t\in (0,T].
\end{align}
Taking the inner product with~\eqref{MBE} by $\frac{\partial u}{\partial t}$, we obtain the energy decay property for the continuous system:
\begin{align*}
\frac{\mbox{d} E}{\mbox{d} t} = - \| u_t \|_{L^2}^2.
\end{align*}
The equation \eqref{MBE} is referred to as the no-slope-selection (NSS) equation, since \eqref{MBE} predicts an unbounded mound slope $m(t) = O(t^{\frac{1}{4}})$ \cite{golubovic1997interfacial,li2003thin}. On the other hand, the slope-selection (SS) case refers to the case when the logarithm term $- \nabla\cdot\left(\frac{\nabla u}{1+|\nabla u|^2}\right)$ in \eqref{MBE} is replaced by $\nabla\cdot(|\nabla u|^2\nabla u)$, which predicts that mound-like or pyramid structures in the surface profile tend to have a uniform, constant mound slope \cite{li2003thin}.

The solution $u(\mathbf{x},t)$ is mass conservative, i.e.,
\begin{align}\label{mass conservation}
\int_{\Omega} u(\mathbf{x},t)~\mbox{d}\mathbf{x} = \int_{\Omega} u(\mathbf{x},0)~\mbox{d}\mathbf{x}, ~ t>0 ,
\end{align}
due to the periodic boundary condition. For simplicity of presentation, herein we assume $u$ to have zero mean-value. Scaling laws that characterize the surface coarsening during the film growth have been a physically interesting problem; see \cite{golubovic1997interfacial, kohn2003upper, li2006high, li2004epitaxial, moldovan2000interfacial}. As for the continuum model, the well-posedness of the initial-boundary-value problem for both SS and NSS equation has been given by \cite{li2006high, li2003thin} through the perturbation analysis and Galerkin spectral approximations. Li and Liu \cite{li2004epitaxial} proved that the minimum energy scales as $\log \varepsilon$ for small $\varepsilon>0$, and the average energy is bounded below by $O(-\log t)$ for large $t$. This implies that long time energy stability is required to simulate the coarsening process. 

One popular way to construct energy stable numerical scheme is to split the energy functional into convex and concave parts, then apply implicit and explicit time stepping algorithms to the corresponding terms~\cite{eyre1998unconditionally}, respectively; see the first such numerical scheme in~\cite{wang2010unconditionally} for the molecular epitaxy growth model. Since then, there have been various works dedicated to deriving high order and energy stable schemes under this idea, such as \cite{chen2012linear, chen2014linear, chen2012mixed, li2017convergence, qiao2015error, qiao2017error, shen2012second}. In particular, a first order in time linear scheme was proposed in~\cite{chen2012linear} to solve the NSS equation, with the energy stability established. On the other hand, the concave term turns out to be nonlinear in this approach, and there has been a well-known difficulty of the convex-splitting approach to construct unconditionally stable higher-order schemes for a nonlinear concave term. Many efforts have been devoted to overcome this subtle difficulty, such as introducing an artificial stabilizing term in the growth equation to balance the explicit treatment of the nonlinear term; see \cite{feng2017second, li2016characterizing, li2018second, xu2006stability}. In addition, there have been some other interesting approaches for the stability of a numerically modified energy functional, such as the invariant energy quadratization (IEQ)  \cite{yang2017numerical} and the scalar auxiliary variable (SAV) methods   \cite{shen2018scalar}.

Other than these direct approaches to establish the energy stability, some other ideas have been reported in recent years to obtain higher order temporal accuracy with explicit treatment of nonlinear terms, such as the time exponential time differencing (ETD) approach. In general, an exact integration of the linear part of the NSS equation is involved in the ETD-based scheme, followed by multi-step explicit approximation of the temporal integral of the nonlinear term~\cite{Suli14, Beylkin98, Cox02, Hochbruck10, Hochbruck11}. An application of such an idea to various gradient models has been reported in recent works~\cite{Wang18, Ju17, Ju14, Ju15a, Ju15b, wangx16, zhu16}, with the high order accuracy and preservation of the exponential behavior observed in the numerical experiments. At the theoretical side, some related convergence analyses have also been reported, while a rigorous energy stability analysis has only been provided for the first order scheme~\cite{Ju17}. For the second and higher order numerical schemes, a theoretical justification of the energy bound has not been available.


In this paper, we work on a second order in time accurate ETD-based scheme, with Fourier pseudo-spectral approximation in space, and the energy stability analysis is established at a theoretical level. We use an alternate approach to overcome the above-mentioned difficulty, instead of applying convex splitting to the energy functional \eqref{MBE energy}. In more details, an artificial stabilizing term $A\tau^2\frac{\partial\Delta^2 u}{\partial t}$ is added in the growth equation, where $A$ is a positive constant and $\tau$ is the time step size. Also, we apply a linear Lagrange approximation to the nonlinear term as in \cite{Ju17}.  
This approach enables us to perform a careful energy estimate, so that a decay property for a modified discrete energy functional is proved. This in turn leads to a uniform in time bound for the original energy functional. In our knowledge, this is the first such result for a second order ETD-based scheme, with the primary difficulty focused on the explicit extrapolation for the nonlinear terms. Moreover, we provide a novel convergence analysis for the proposed scheme. Instead of analyzing the operator form of the numerical error function, here we start from the continuous in time ODE system satisfied by the error function, which is derived from the corresponding equations of the numerical solution. With a careful treatment of the aliasing error and $H^3$ estimate of the numerical solution, we are able to derive an $\ell^{\infty}(0,T; \ell^2)$ error estimate for the proposed scheme.

There have been quite a few physically interesting quantities that may be obtained from the solutions of the NSS equation~\eqref{MBE}, such as the energy, average surface roughness and the average slope. A theoretical analysis in~\cite{li2004epitaxial} has implied a lower bound for the energy dissipation as the order of $O (- \ln t)$, and an upper bound for the average roughness as the order of $O (t^{1/2})$, which in turn implies an $O (t^{1/4})$ order for the average slope. In particular, the  $O (- \ln t)$ energy dissipation scale leads to a long time scale nature of the NSS equation~\eqref{MBE}, in comparison with the $O (t^{-1/3}$ scale for the slope-selection version and the standard Cahn-Hilliard model. In addition, a lower bound for the energy, estimated as $\gamma:= \frac{L^2}{2}  \left( \ln (4 \varepsilon^2 L^{-2} ) - 4 \varepsilon^2 \pi^2 + 1 \right)$ as derived in~\cite{chen2012linear, chen2014linear, wang2010unconditionally}, indicates an intuitive $O (\varepsilon^{-2})$ law for the saturation time scale for $\varepsilon \ll\min L$, because of the $O (- \ln t)$ energy dissipation scale. All these facts have demonstrated the necessity of energy stable numerical approach to accurately capture the physically interesting quantities, since the energy stability turns out to be a global in time nature for the numerical scheme. In the extensive numerical experiments, the proposed stabilized ETD scheme has produced highly accurate solutions, which obtain the long time asymptotic growth rate of the surface roughness and average slope with relative accuracy within $4\%$.

The rest of this article is organized as follows. In Section~\ref{sec: sETDMs2 scheme} we present the fully discrete numerical scheme. The numerical energy stability is proved in Section~\ref{sec: sETDMs2 stab}, followed by $\ell^{\infty}(0,T; H_h^2)$ and $\ell^{\infty}(0,T; H_h^3)$ bounds of the numerical solution. Subsequently, an $\ell^{\infty}(0,T; \ell^2)$ error analysis is given by Section~\ref{sec: error analysis}, consisting of two lemmas concerning the error of the nonlinear term. In addition, numerical experiments are provided in Section~\ref{sec: numerical results}, including temporal convergence test and simulation results of the scaling laws for energy, average surface roughness and average slope. Finally, some concluding remarks are given by Section~\ref{sec:conclusion}.

\section{The stabilized second order in time ETD mulstistep scheme (sEDTMs2)}\label{sec: sETDMs2 scheme}

Some definitions in \cite{adams2003} are recalled. Consider $\Omega = [0,L]^d$. We define $W^{m,p}(\Omega):=\{v\in L^p(\Omega)\mid D^{\alpha}v\in L^p(\Omega)\text{ for }0\leq |\alpha|\leq m\}$, where $\alpha = (\alpha_1,\ldots,\alpha_d)$ is a $d$-tuple of non-negative integers with $|\alpha| = \sum_{i=1}^d\alpha_i$, and $D_{\alpha} = \prod_{i=1}^d D_{x_i}^{\alpha_i}$. For simplicity, we denote $\|\cdot\|_{W^{m,p}(\Omega)}$ as $\|\cdot\|_{m,p}$, and the related semi-norm as $|\cdot|_{m,p}$. In particular, if $p=2$, we set $W^{m,p}(\Omega)$ as $H^m(\Omega)$, and $\|\cdot\|_{m,2}$ as $\|\cdot\|_{H^m}$.

In this paper, we focus on the case when $d=2$ and follow the notations used in collocation spectral method; see \cite{canuto2006spectrala, canuto2007spectralb, gottlieb1977numerical, gottlieb2012stability, Ju17, shen2011spectral}. Let $N$ be a positive integer, $\Omega_\mathcal{N}$ be a uniform $2N\times 2N$ mesh on $\Omega$, with $(2N+1)^2$ grid points $(x_i,y_j)$, where $x_i = ih$, $y_j = jh$ with $h:=\frac{L}{2N}$, $0\leq i,j\leq 2N$. Denote $u_e$ as the exact solution of \eqref{MBE}, $u(t) = u_e(\cdot,t)|_{\Omega_{\mathcal{N}}}$ be the restriction of the exact solution on $\Omega_{\mathcal{N}}$, $\tau$ be the time step size $\frac{T}{N_t}$, and $t_i = i\tau$ for $0\leq i \leq N_t$. Let $H^m_{per}(\Omega)=\{v\in H^m(\Omega)\mid v\text{ is periodic} \}$, $\mathcal{M}^{\mathcal{N}}$ be the space of 2D periodic grid functions on $\Omega_\mathcal{N}$, and $\mathcal{B}^{N}$ be the space of trigonometric polynomials in $x$ and $y$ of degree up to $N$. In this paper, we denote $C$ one generic constant which may depend on $\varepsilon$, the solution $u$, the initial value $u_0$ and time $T$, but is independent of the mesh size $h$ and time step size $\tau$.

Let us firstly recall the following Berstein inequality introduced in \cite[p. 33, Lemma 2.5]{shen2011spectral}:
\begin{lem}\label{inverse estimate}
For any $u\in \mathcal{B}^{N}$ and $1\leq p\leq\infty$, we have
\begin{equation}
\|\partial_x^m u\|_{L^p(\Omega)}\leq CN^{m}\|u\|_{L^p(\Omega)}, \quad m\geq 1.
\label{inverse estimate-0}
\end{equation}
\end{lem}

Now we introduce an interpolation operator $\mathcal{I}_N$ onto $\mathcal{B}^{N}$ that reserves the function value on $(2N+1)^2$ grid points, i.e., $(\mathcal{I}_Nf)(x_i, y_j) = f(x_i, y_j)$ for $0\leq i,j\leq 2N$:
\begin{align}\label{interpolation operator}
(\mathcal{I}_N f)(x,y) = \sum_{k,l = -N+1}^{N}(\hat{f}_c)_{k,l}\exp\left(\frac{2\pi\mathrm{i}}{L}(kx+ly)\right), \quad \mbox{with $\mathrm{i} = \sqrt{-1}$} ,
\end{align}
where the coefficients $\{(\hat{f}_c)_{k,l}\}$ are given by the discrete Fourier transform of the $4N^2$ grid points:
\begin{align}
(\hat{f}_c)_{k,l} = \frac{1}{4N^2}\sum_{i,j = 1}^{2N} f(x_i,y_j)\exp\left(-\frac{2\pi\mathrm{i}}{L}(kx_i+ly_j)\right).
\end{align}
For any $f\in \mathcal{M}^{\mathcal{N}}$, denote $\tilde{f} = \mathcal{I}_N f$ as the continuous extension of $f$. When $f$ and $\partial^{\alpha}f$ ($|\alpha|\leq m$) are continuous and periodic on $\Omega$, $\mathcal{I}_N$ has the following approximation property (\cite[Theorem 1.2, p. 72]{canuto1982approximation}):
\begin{equation}\label{I_N convergence}
\|\partial^k f - \partial^k \mathcal{I}_Nf\|_{L^2} \leq Ch^{m-k}\|f\|_{H^m}, ~\text{for }0\leq k\leq m, ~m>\frac{d}{2},
\end{equation}
for dimension $d$. Also, the following $H^m$ bound of $\mathcal{I}_N$ is excerpted in \cite[Lemma 1]{gottlieb2012stability}:
\begin{lem}\label{I_N Hm bound}
For any $\varphi\in\mathcal{B}^{2N}$ in dimension $d$, we have
\begin{align}
\|\mathcal{I}_N\varphi\|_{H^k} \leq (\sqrt{2})^d\|\varphi\|_{H^k}, ~k\geq 0.
\end{align}
\end{lem}

Given $\mathcal{I}_Nf$, the discrete spatial partial derivatives can be defined as:
\begin{align*}
(D_x f)_{i,j} &= \sum_{k,l = -N+1}^{N}\frac{2k\pi\mathrm{i}}{L}(\hat{f}_c)_{k,l}\exp\left(\frac{2\pi\mathrm{i}}{L}(kx_i+ly_j)\right),\\
(D^2_x f)_{i,j} &= \sum_{k,l = -N+1}^{N} -\frac{4k^2\pi^2}{L^2}(\hat{f}_c)_{k,l}\exp\left(\frac{2\pi\mathrm{i}}{L}(kx_i+ly_j)\right).
\end{align*}
Similarly we have $D_y$ and $D^2_y$. For any $f, ~g\in \mathcal{M}^\mathcal{N}$, and $\mathbf{f} = (f_1,f_2)^T, ~\mathbf{g} = (g_1,g_2)^T\in\mathcal{M}^\mathcal{N}\times\mathcal{M}^\mathcal{N}$, we can define the discrete gradient, divergence and Laplace operators:
\begin{equation*}
\nabla_{\mathcal{N}}f=\left(
\begin{matrix}
D_x f\\
D_y f\\
\end{matrix}
\right), \quad\nabla_{\mathcal{N}}\cdot f = D_x f^1+D_y f^2, \quad\Delta_{\mathcal{N}}f=D^2_x f+D^2_y f.
\end{equation*}
Also, to measure the discrete differentiation operators defined above, we introduce the discrete $L^2$ (denoted as $\ell^2$) inner product $(\cdot,\cdot)_{\mathcal{N}}$ and norm $\|\cdot\|_{\mathcal{N}}$:
\begin{align*}
&(f,g)_{\mathcal{N}} = h^2\sum_{i,j=1}^{2N} f_{ij}g_{ij}, \qquad \|f\|_{\mathcal{N}}=\sqrt{(f,f)_{\mathcal{N}}},\\
&(\mathbf{f},\mathbf{g})_{\mathcal{N}} = h^2\sum_{i,j=1}^{2N} (f^1_{ij}g^1_{ij}+f^2_{ij}g^2_{ij}), \qquad \|\mathbf{g}\|_{\mathcal{N}}=\sqrt{(\mathbf{f},\mathbf{f})_{\mathcal{N}}}.
\end{align*}
Similarly, we can define the discrete Sobolev norm $\|\cdot\|_{H^2_h}$ and the discrete Sobolev semi-norm $|\cdot|_{H^2_h}$:
\begin{align*}
\|f\|_{H^2_h} = \Bigl( \sum_{|\alpha|\leq 2}\|D_{\alpha}f\|_{\mathcal{N}}^2 \Bigr)^{\frac{1}{2}}, \quad |f|_{H^2_h} = \Bigl( \sum_{|\alpha|= 2}\|D_{\alpha}f\|_{\mathcal{N}}^2 \Bigr)^{\frac{1}{2}},
\end{align*}
where as above $\alpha = (\alpha_1,\alpha_2)$ is a $2$-tuple of nonnegative integers with $|\alpha| = \alpha_1 + \alpha_2$, and $D_{\alpha} = D_x^{\alpha_1}D_y^{\alpha_2}.$ Furthermore, the following summation by parts formulas are available in \cite[proposition 2.2]{Ju17}:
\begin{lem}\label{IBP}
For any $f, ~g\in \mathcal{M}^\mathcal{N}$, and $\mathbf{f} = (f_1,f_2)^T, ~\mathbf{g} = (g_1,g_2)^T\in\mathcal{M}^\mathcal{N}\times\mathcal{M}^\mathcal{N}$, we have the following summation by parts formula:
\begin{equation*}
(f,\nabla_{\mathcal{N}}\cdot \mathbf{g})_{\mathcal{N}}=-(\nabla_{\mathcal{N}}f,\mathbf{g})_{\mathcal{N}}, \quad (f,\Delta_{\mathcal{N}}g)_{\mathcal{N}}=-(\nabla_{\mathcal{N}}f,\nabla_{\mathcal{N}}g)_{\mathcal{N}} = (\Delta_{\mathcal{N}}f, g)_{\mathcal{N}}.
\end{equation*}
\end{lem}

Recall that the exact solution $u$ in \eqref{MBE} is assumed to be mean value free, thus we consider the subset of zero-mean grid functions:
\begin{equation*}
{\mathcal{M}}^{\mathcal{N}}_0 = \{v\in {\mathcal{M}}^{\mathcal{N}}\mid (v,1)_{\mathcal{N}}=0\}=\{v\in {\mathcal{M}}^{\mathcal{N}}\mid \hat{v}_{00}=0\}.
\end{equation*}
Define $L_{\mathcal{N}} = \varepsilon^2\Delta_{\mathcal{N}}^2|_{{\mathcal{M}}^{\mathcal{N}}_0}$, which is symmetric positive definite on ${\mathcal{M}}_0^{\mathcal{N}}$.

In turn, the spatial discretization of \eqref{MBE} becomes: Given $u_0\in {\mathcal{M}}^{\mathcal{N}}_0$, find $\hat{u}: [0,T]\rightarrow {\mathcal{M}}^{\mathcal{N}}_0$ such that
\begin{equation}\label{MBE spatial discrete}
\begin{aligned}
&\frac{\mbox{d}\hat{u}}{\mbox{d}t} = - L_{\mathcal{N}}\hat{u}-f_{\mathcal{N}}(\hat{u}),
  \quad L_{\mathcal{N}} = \varepsilon^2 \Delta_{\mathcal{N}}^2 ,  \quad t\in(0,T],\\
&\hat{u}(0) = u(0),
\end{aligned}
\end{equation}
where $f_{\mathcal{N}}(\hat{u})=\nabla_{\mathcal{N}}\cdot\left(\frac{\nabla_{\mathcal{N}}\hat{u}}{1+|\nabla_{\mathcal{N}}\hat{u}|^2}\right)$. Multiplying both sides of \eqref{MBE spatial discrete} by $e^{L_{\mathcal{N}}t}$ yields
\begin{equation}\label{derive step: eLt}
\frac{\mbox{d}e^{L_{\mathcal{N}}t}\hat{u}}{\mbox{d}t}=-e^{L_{\mathcal{N}}t}f_{\mathcal{N}}(\hat{u}).
\end{equation}
Integrating \eqref{derive step: eLt} from $t_n$ to $t_{n+1}$ gives
\begin{equation}\label{MBE spatial discrete2}
\hat{u}(t_{n+1}) = e^{-L_{\mathcal{N}}\tau}\hat{u}(t_n) - \int^{t_{n+1}}_{t_n}e^{-L_{\mathcal{N}}(t_{n+1}-t)}f_{\mathcal{N}}(\hat{u}(t))\mbox{d}t.
\end{equation}
By \cite[Lemma 4.1]{Ju17}, the error between the exact solution $u(t)$ and the solution $\hat{u}(t)$ of \eqref{MBE spatial discrete2} is of order $O(N^{-m})$, given that $u_e(t)$ is smooth enough.

\subsection{The ETD1 and ETDMs2 schemes}
Two exponential time differencing schemes (ETD1 and ETDMs2) have been proposed in~\cite{Ju17}, using the energy convex splitting method. Ju et al used a similar spatial discretization as \eqref{MBE spatial discrete2} is derived, with modified $L_{\mathcal{N}}$ and $f_{\mathcal{N}}(\hat{u}(t))$:
\begin{align}
L_{c} &= L_{\mathcal{N}} - \kappa\Delta_{\mathcal{N}},  \quad \kappa > 0 , \label{Lc}\\
f_{e}(\hat{u}(t)) &= f_{\mathcal{N}}(\hat{u}(t)) + \kappa\Delta_{\mathcal{N}}\hat{u}(t) . \label{fe}
\end{align}
For the ETD1, the term $f_{e}(\hat{u}(t))$ in $[t_n,t_{n+1}]$ is simply approximated by $f_{e}(\hat{u}(t_n))$. For the ETDMs2, $f_{e}(\hat{u}(t))$ is approximated by the linear Lagrange interpolation:
\begin{align}
 f_{e}(\hat{u}(t_{n})) + \frac{t-t_n}{\tau}[f_{e}(\hat{u}(t_{n})) - f_{e}(\hat{u}(t_{n-1}))], ~t\in [t_n, t_{n+1}]. \notag
\end{align}
Let $u_h(t)$ be the numerical solution of ETD1 and ETDMs2. Denote $u_h(t_n)$ as $u_h^{n}$ for $n\geq 0$. Integrating from $t_n$ to $t_{n+1}$, the following expressions are obtained in~\cite{Ju17}:
\begin{itemize}
\item \bf{ETD1:} 
$$u_h^{n+1} = e^{-\tau L_{c}}u_h^n - \phi_0(L_{c})f_{e}(u_h^n).$$
\item \bf{ETDMs2:} $$u_h^{n+1} = e^{-\tau L_{c}}u_h^n - \phi_0(L_{c})f_{e}(u_h^n) - \phi_1(L_{c})[f_{e}(u_h^n) - f_{e}(u_h^{n-1})],
$$
\end{itemize}
in which
\begin{align*}
&\phi_0(x) = x^{-1}(I - e^{-x\tau}),
&\phi_1(x) = x^{-1}[1 - (x\tau)^{-1}(I - e^{-x\tau})].
\end{align*}
The convergence analysis for both ETD1 and ETDMs2 was given by~\cite{Ju17}. However, a theoretical proof of the long time energy stability of ETDMs2 was not provided.

\subsection{The sETDMs2 scheme}
In this section we propose a second order in time stabilized exponential time differencing multistep scheme (sETDMs2). In order to guarantee the long time energy stability, a stabilizing term $A\tau^2\frac{\mbox{d} \Delta^2_{\mathcal{N}} \hat{u}(t)}{\mbox{d} t}$ is introduced, in which $A$ is a positive constant.  Also, we apply the same Lagrange approximation for $f_{\mathcal{N}}$ as in~\cite{Ju17}. Define a continuous in time function for $s\in [0, \tau]$:
\begin{align*}
f_{\mathcal{N},L}(s, \hat{u}(t_{n}),\hat{u}(t_{n-1})):= f_{\mathcal{N}}(\hat{u}(t_{n})) + \frac{s}{\tau}[f_{\mathcal{N}}(\hat{u}(t_{n})) - f_{\mathcal{N}}(\hat{u}(t_{n-1}))].
\end{align*}
We present the following sETDMs2 scheme:
\begin{itemize}
\item {\bfseries sETDMs2: } For $n\geq 1$, find $u_s^{n+1}: [0, t_{n+1}]\rightarrow {\mathcal{M}}^{\mathcal{N}}_0$ such that for any $t\in (t_n, t_{n+1}]$,
\begin{align}
&~\frac{\mbox{d} u_s^{n+1}(t)}{\mbox{d} t} + A\tau^2\frac{\mbox{d}\Delta^2_{\mathcal{N}} u_s^{n+1}(t)}{\mbox{d} t}= -L_{\mathcal{N}} u_s^{n+1}(t) - f_{\mathcal{N},L}(t - t_n, u_s^n,u_s^{n-1}),\label{sETDMs2: discrete}
\end{align}
with $u_s^{n+1}(t) := u_s^{n}(t)$ for $t\in [0, t_{n}]$.
\end{itemize}
The initial step is obtained as follow: Find $u_s^{1}: ~[0,t_{1}]\rightarrow {\mathcal{M}}^{\mathcal{N}}_0$ such that
\begin{align}
\frac{\mbox{d} u_s^{1}(t)}{\mbox{d} t} + A\tau^2\frac{\mbox{d} \Delta^2_{\mathcal{N}} u_s^{1}(t)}{\mbox{d} t}
= -L_{\mathcal{N}} u_s^{1}(t) - f_{\mathcal{N}}(u_s^{0}), ~t\in (0,t_1],\label{sETDMs2: discrete-init}
\end{align}
where $u_s^{0}= u(0)$.
In fact, as for the initial step, the stabilizing term can also be replaced by $-A\tau\frac{\mbox{d} \Delta_{\mathcal{N}} u_s^{1}(t)}{\mbox{d} t}$. Integrating \eqref{sETDMs2: discrete} and \eqref{sETDMs2: discrete-init} from $t_n$ to $t_{n+1}$, one obtains explicit expressions of $u_s^1$ and $u_s^{n+1}$ for sETDMs2:
\begin{itemize}
\item {\bfseries sETDMs2 (matrix form): }  For $n\ge 1$,
\end{itemize}
\begin{equation}
u_s^{n+1} = e^{-\mathcal{K}_{\mathcal{N}}\tau}u_s^n - \phi_0(\mathcal{K}_{\mathcal{N}})\mathcal{G}_{\mathcal{N}}(u_s^n) - \phi_1(\mathcal{K}_{\mathcal{N}})[\mathcal{G}_{\mathcal{N}}(u_s^n) - \mathcal{G}_{\mathcal{N}}(u_s^{n-1})],\label{sETDMs2}
\end{equation}
and the intial step is computed by
\begin{equation}
u_s^1 = e^{-\mathcal{K}_{\mathcal{N}}\tau}u_s^0 - \phi_0(\mathcal{K}_{\mathcal{N}})\mathcal{G}_{\mathcal{N}}(u_s^0),\label{sETDMs2-init}
\end{equation}
in which
\begin{align*}
&\mathcal{K}_{\mathcal{N}} = \varepsilon^2 (I+A\tau^2\Delta_{\mathcal{N}}^2)^{-1}\Delta_{\mathcal{N}}^2,
&\mathcal{G}_{\mathcal{N}}(v) = (I+A\tau^2\Delta_{\mathcal{N}}^2)^{-1}f_{\mathcal{N}}(v),\\
&\phi_0(\mathcal{K}_{\mathcal{N}}) = \mathcal{K}_{\mathcal{N}}^{-1}(I - e^{-\mathcal{K}_{\mathcal{N}}\tau}),
&\phi_1(\mathcal{K}_{\mathcal{N}}) = \mathcal{K}_{\mathcal{N}}^{-1}[I - (\mathcal{K}_{\mathcal{N}}\tau)^{-1}(I - e^{-\mathcal{K}_{\mathcal{N}}\tau})].
\end{align*}
Note that sETDMs2 and ETDMs2 have similar matrix forms for solutions at $\{t_i\}_{i=1}^{N_t}$, and they share similar computational complexity.

\section{Long time energy stability of the sETDMs2 scheme}\label{sec: sETDMs2 stab}
In later analysis we shall make frequent use of the Gagliardo-Nirenberg inequality and Agmon's inequality (see~\cite{adams2003, agmon2010lectures, Nirenberg1959on}): for function $v$ in $\mathbb{R}^d$,
\begin{align}
|v|_{W^{j,p}(\Omega)}&\leq C\|v\|_{L^q(\Omega)}^{1-\theta}\|v\|_{W^{m,r}(\Omega)}^{\theta}, \quad \frac{1}{p} = \frac{j}{d}+\theta \left(\frac{1}{r}-\frac{m}{d}\right)+(1-\theta)\frac{1}{q},\label{Gagliardo-Nirenberg}\\
\|v\|_{L^{\infty}(\Omega)}&\leq C\|v\|_{L^2(\Omega)}^{1-\frac{n}{2m}}\|v\|_{H^m(\Omega)}^{\frac{n}{2m}}, \quad m>\frac{d}{2},~ v\in H^m(\Omega) , \label{Agmon}
\end{align}
with $1\leq q,r\leq \infty, ~0\leq j<m, ~\frac{j}{m}\leq \theta\leq 1$.

Consider the following numerical energy functional:
\begin{align}
\tilde{E}(u_s^{n+1}(t))&= E_{\mathcal{N}}(u_s^{n+1}(t)) +\frac{\sqrt{3}-1}{4} \left\|\frac{\mbox{d} u_s^{n+1}(t)}{\mbox{d} t}\right\|_{L^2(t_n,t_{n+1}; \ell^2)}^2 \nonumber\\
&\quad + \frac{(\sqrt{3}+1)\tau^2}{24}\left\|\frac{\mbox{d} \Delta_{\mathcal{N}} u_s^{n+1}(t)}{\mbox{d} t}\right\|_{L^2(t_n,t_{n+1}; \ell^2)}^2,
\end{align}
where $E_{\mathcal{N}}(v)$ is the discrete version of the continuous energy functional $E(v)$ in \eqref{MBE energy}:
\begin{equation}\label{MBE energy: discrete}
E_{\mathcal{N}}(v)=\left(-\frac{1}{2}\ln(1+|\nabla_{\mathcal{N}}v|^2),1\right)_{\mathcal{N}} + \frac{\varepsilon^2}{2}\|\Delta_{\mathcal{N}} v\|_{\mathcal{N}}^2,~~~\forall v \in {\mathcal{M}^{\mathcal{N}}}.
\end{equation}

The main theoretical result of this section is the following theorem.
\begin{thm}\label{thm: energy stab}
Assume that $A\geq \frac{2+\sqrt{3}}{6}$. The numerical system \eqref{sETDMs2: discrete} is energy stable in the sense that for any $1 \leq n\leq N_t-1$,
\begin{align*}
\tilde{E}(u_s^{n+1}) + \left(A-\frac{2+\sqrt{3}}{6}\right)\tau^2\left\|\frac{{\rm d} \Delta_{\mathcal{N}} u_s^{n+1}(t)}{{\rm d} t}\right\|_{L^2(t_n,t_{n+1}; \ell^2)}^2 \leq \tilde{E}(u_s^{n}) .
\end{align*}
\end{thm}

\begin{proof}
We define $\beta(\mathbf{v}) := \frac{\mathbf{v}}{1+|\mathbf{v}|^2}$ and denote
$$
\mathbf{\hat{f}}_{\mathcal{N},L}(t - t_n, u_s^{n},u_s^{n-1}) := \beta(\nabla_{\mathcal{N}} u_s^{n}) + \frac{t - t_n}{\tau}[\beta(\nabla_{\mathcal{N}} u_s^{n}) - \beta(\nabla_{\mathcal{N}} u_s^{n-1})].
$$
By definition we have: For any $v, ~w\in\mathcal{M}^{\mathcal{N}}$ and $\zeta\geq 0$,
\begin{align*}
&f_{\mathcal{N}}(v) = \nabla_{\mathcal{N}}\cdot\beta(\nabla_{\mathcal{N}}v),~f_{\mathcal{N},L}(\zeta, v, w) =  \nabla_{\mathcal{N}}\cdot\mathbf{\hat{f}}_{\mathcal{N},L}(\zeta, v, w).
\end{align*}
Taking an inner product with \eqref{sETDMs2: discrete} by $\frac{\mbox{d} u_s^{n+1}(t)}{\mbox{d} t}$ gives that: For $n\geq 1$,
\begin{align}
&\left\|\frac{\mbox{d} u_s^{n+1}(t)}{\mbox{d} t}\right\|_{\mathcal{N}}^2 + A\tau^2 \left\|\frac{\mbox{d} \Delta_{\mathcal{N}} u_s^{n+1}(t)}{\mbox{d} t}\right\|_{\mathcal{N}}^2 + \frac{\varepsilon^2}{2}\frac{{\rm d} \|\Delta_{\mathcal{N}} u_s^{n+1}(t)\|_{\mathcal{N}}^2}{{\rm d} t} \nonumber \\
= ~& \left(\mathbf{\hat{f}}_{\mathcal{N},L}(t - t_n, u_s^{n},u_s^{n-1}), \frac{{\rm d} \nabla_{\mathcal{N}} u_s^{n+1}(t)}{{\rm d} t} \right)_{\mathcal{N}}.\label{sETDMs2: middle1}
\end{align}
Recall that $E_{\mathcal{N}}(u)=\left(-\frac{1}{2}\ln(1+|\nabla_{\mathcal{N}}u|^2),1\right)_{\mathcal{N}} + \frac{\varepsilon^2}{2}\|\Delta_{\mathcal{N}} u\|_{\mathcal{N}}^2$, thus
$$
\frac{{\rm d}}{{\rm d}t}E_{\mathcal{N}}(u_s^{n+1}(t)) = \left(-\beta(\nabla_{\mathcal{N}}u_s^{n+1}(t)), \frac{{\rm d} \nabla_{\mathcal{N}} u_s^{n+1}(t)}{{\rm d} t} \right)_{\mathcal{N}} + \frac{\varepsilon^2}{2}\frac{{\rm d} \|\Delta_{\mathcal{N}} u_s^{n+1}(t)\|_{\mathcal{N}}^2}{{\rm d} t}.
$$
Therefore, we can rewrite \eqref{sETDMs2: middle1} as below:
\begin{align}
&\left\|\frac{\mbox{d} u_s^{n+1}(t)}{\mbox{d} t}\right\|_{\mathcal{N}}^2 + A\tau^2 \left\|\frac{\mbox{d} \Delta_{\mathcal{N}} u_s^{n+1}(t)}{\mbox{d} t}\right\|_{\mathcal{N}}^2 + \frac{{\rm d}}{{\rm d}t}E_{\mathcal{N}}(u_s^{n+1}(t)) \nonumber \\
= ~& \left(\mathbf{\hat{f}}_{\mathcal{N},L}(t - t_n, u_s^{n},u_s^{n-1}) - \beta(\nabla_{\mathcal{N}}u_s^{n+1}(t)), \frac{{\rm d} \nabla_{\mathcal{N}} u_s^{n+1}(t)}{{\rm d} t} \right)_{\mathcal{N}}.\label{sETDMs2: middle2}
\end{align}
Integrating \eqref{sETDMs2: middle2} from $t_n$ to $t_{n+1}$, we have
\begin{align}
&\left\|\frac{\mbox{d} u_s^{n+1}(t)}{\mbox{d} t}\right\|_{L^2(t_n,t_{n+1}; \ell^2)}^2 + A\tau^2 \left\|\frac{\mbox{d} \Delta_{\mathcal{N}} u_s^{n+1}(t)}{\mbox{d} t}\right\|_{L^2(t_n,t_{n+1}; \ell^2)}^2 \nonumber \\
& + E_{\mathcal{N}}(u_s^{n+1}) - E_{\mathcal{N}}(u_s^{n}) \nonumber \\
= ~&\int_{t_n}^{t_{n+1}}\left(\beta(\nabla_{\mathcal{N}} u_s^n) - \beta(\nabla_{\mathcal{N}} u_s^{n+1}(t)), \frac{{\rm d} \nabla_{\mathcal{N}} u_s^{n+1}(t)}{{\rm d} t} \right)_{\mathcal{N}} ~dt \nonumber \\
& + \int_{t_n}^{t_{n+1}} \frac{t-t_n}{\tau}\left(\beta(\nabla_{\mathcal{N}} u_s^n) - \beta(\nabla_{\mathcal{N}} u_s^{n-1}), \frac{{\rm d} \nabla_{\mathcal{N}} u_s^{n+1}(t)}{{\rm d} t} \right)_{\mathcal{N}} ~dt \nonumber \\
:= ~& (\uppercase\expandafter{\romannumeral1}) + (\uppercase\expandafter{\romannumeral2}).\label{sETDMs2: middle3}
\end{align}
As for $(\uppercase\expandafter{\romannumeral1})$, an application of H$\rm\ddot{o}$lder's inequality gives
\begin{align}
(\uppercase\expandafter{\romannumeral1}) &\leq \int_{t_n}^{t_{n+1}}\|\beta(\nabla_{\mathcal{N}} u_s^{n+1}(t)) - \beta(\nabla_{\mathcal{N}} u_s^n)\|_{\mathcal{N}} \left\|\frac{{\rm d} \nabla_{\mathcal{N}} u_s^{n+1}(t)}{{\rm d} t}\right\|_{\mathcal{N}}~dt.\label{(I) est1}
\end{align}
Using the inequality $|\beta(\mathbf{v}) - \beta(\mathbf{w})| \leq |\mathbf{v} - \mathbf{w}|$ in \cite[Lemma 3.5]{Ju17} and reversing the order of integration and summation, we obtain
\begin{align}
\|\beta(\nabla_{\mathcal{N}} u_s^{n+1}(t)) - \beta(\nabla_{\mathcal{N}} u_s^n)\|_{\mathcal{N}}
&\leq \| \nabla_{\mathcal{N}} u_s^{n+1}(t) - \nabla_{\mathcal{N}} u_s^{n} \|_{\mathcal{N}} \nonumber\\
&= \left\|\int_{t_n}^{t} \frac{{\rm d}\nabla_{\mathcal{N}} u_s^{n+1}(l)}{{\rm d}t} ~dl\right\|_{\mathcal{N}} \nonumber \\
&\leq \tau^{\frac{1}{2}}\left\|\frac{{\rm d}\nabla_{\mathcal{N}} u_s^{n+1}(t)}{{\rm d}t}\right\|_{L^2(t_n,t_{n+1}; \ell^2)}.\label{(I) est2}
\end{align}
Substituting \eqref{(I) est2} into \eqref{(I) est1} and applying H$\rm\ddot{o}$lder's inequality again, we get
\begin{align}
(\uppercase\expandafter{\romannumeral1}) &\leq \tau\left\|\frac{{\rm d}\nabla_{\mathcal{N}} u_s^{n+1}(t)}{{\rm d}t}\right\|_{L^2(t_n,t_{n+1}; \ell^2)}^2.\label{I nabla}
\end{align}
Also notice that by the interpolation inequality, we have
$$
\left\|\frac{{\rm d} \nabla_{\mathcal{N}} u_s^{n+1}(t)}{{\rm d} t}\right\|_{\mathcal{N}}^2 \leq \frac{1}{2\tau \lambda}\left\|\frac{\mbox{d} u_s^{n+1}(t)}{\mbox{d} t}\right\|_{\mathcal{N}}^2 + \frac{\tau \lambda}{2}\left\|\frac{\mbox{d} \Delta_{\mathcal{N}} u_s^{n+1}(t)}{\mbox{d} t}\right\|_{\mathcal{N}}^2,
$$
where $\lambda$ is a positive constant to be decided later. Thus we arrive:
\begin{align}
(\uppercase\expandafter{\romannumeral1}) &\leq \frac{1}{2\lambda}\left\|\frac{\mbox{d} u_s^{n+1}(t)}{\mbox{d} t}\right\|_{L^2(t_n,t_{n+1}; \ell^2)}^2 + \frac{\tau^2 \lambda}{2}\left\|\frac{\mbox{d} \Delta_{\mathcal{N}} u_s^{n+1}(t)}{\mbox{d} t}\right\|_{L^2(t_n,t_{n+1}; \ell^2)}^2.\label{sETDMs2: est1}
\end{align}

The term $(\uppercase\expandafter{\romannumeral2})$ could be estimated in a similar way:
\begin{align}
(\uppercase\expandafter{\romannumeral2})
&\leq \int_{t_n}^{t_{n+1}}\frac{t-t_n}{\tau}\|\nabla_{\mathcal{N}} u_s^n - \nabla_{\mathcal{N}} u_s^{n-1}\|_{\mathcal{N}} \left\|\frac{{\rm d} \nabla_{\mathcal{N}} u_s^{n+1}(t)}{{\rm d} t}\right\|_{\mathcal{N}}~dt \nonumber \\
&\leq \tau^{\frac{1}{2}} \left\|\frac{{\rm d} \nabla_{\mathcal{N}} u_s^n(t)}{{\rm d} t}\right\|_{L^2(t_{n-1},t_{n}; \ell^2)}\int_{t_n}^{t_{n+1}} \frac{t-t_n}{\tau} \left\|\frac{{\rm d} \nabla_{\mathcal{N}} u_s^{n+1}(t)}{{\rm d} t}\right\|_{\mathcal{N}}~dt \nonumber \\
&\leq \frac{\tau}{\sqrt{3}} \left\|\frac{{\rm d} \nabla_{\mathcal{N}} u_s^{n}(t)}{{\rm d} t}\right\|_{L^2(t_{n-1},t_{n}; \ell^2)} \left\|\frac{{\rm d} \nabla_{\mathcal{N}} u_s^{n+1}(t)}{{\rm d} t}\right\|_{L^2(t_n,t_{n+1}; \ell^2)} \nonumber \\
&\leq \frac{\tau}{2\sqrt{3}} \left\|\frac{{\rm d} \nabla_{\mathcal{N}} u_s^{n}(t)}{{\rm d} t}\right\|_{L^2(t_{n-1},t_{n}; \ell^2)}^2 + \frac{\tau}{2\sqrt{3}}\left\|\frac{{\rm d} \nabla_{\mathcal{N}} u_s^{n+1}(t)}{{\rm d} t}\right\|_{L^2(t_n,t_{n+1}; \ell^2)}^2,\label{II nabla}
\end{align}
where we have applied the Cauchy-Schwarz inequality to obtain the last inequality. Again, we make use of the interpolation inequality and get
\begin{align}
(\uppercase\expandafter{\romannumeral2}) &\leq \frac{1}{4\sqrt{3}\lambda} \left\|\frac{{\rm d} u_s^{n}(t)}{{\rm d} t}\right\|_{L^2(t_{n-1},t_{n}; \ell^2)}^2 +\frac{\tau^2 \lambda}{4\sqrt{3}} \left\|\frac{{\rm d} \Delta_{\mathcal{N}} u_s^{n}(t)}{{\rm d} t}\right\|_{L^2(t_{n-1},t_{n}; \ell^2)}^2  \nonumber\\
&+ \frac{1}{4\sqrt{3}\lambda}\left\|\frac{\mbox{d} u_s^{n+1}(t)}{\mbox{d} t}\right\|_{L^2(t_n,t_{n+1}; \ell^2)}^2+ \frac{\tau^2 \lambda}{4\sqrt{3}}\left\|\frac{\mbox{d} \Delta_{\mathcal{N}} u_s^{n+1}(t)}{\mbox{d} t}\right\|_{L^2(t_n,t_{n+1}; \ell^2)}^2.\label{sETDMs2: est2}
\end{align}

Substituting \eqref{sETDMs2: est1} and \eqref{sETDMs2: est2} into \eqref{sETDMs2: middle3}, we arrive at
\begin{align}
0\geq &\left(1-\frac{2\sqrt{3}+1}{4\sqrt{3}\lambda}\right)\left\|\frac{\mbox{d} u_s^{n+1}(t)}{\mbox{d} t}\right\|_{L^2(t_n,t_{n+1}; \ell^2)}^2 - \frac{1}{4\sqrt{3}\lambda} \left\|\frac{{\rm d} u_s^{n}(t)}{{\rm d} t}\right\|_{L^2(t_{n-1},t_{n}; \ell^2)}^2 \nonumber \\
&+\left(A-\frac{(2\sqrt{3}+1)\lambda}{4\sqrt{3}} \right)\tau^2\left\|\frac{\mbox{d} \Delta_{\mathcal{N}} u_s^{n+1}(t)}{\mbox{d} t}\right\|_{L^2(t_n,t_{n+1}; \ell^2)}^2  \nonumber \\
& -\frac{\lambda\tau^2}{4\sqrt{3}} \left\|\frac{{\rm d} \Delta_{\mathcal{N}} u_s^{n}(t)}{{\rm d} t}\right\|_{L^2(t_{n-1},t_{n}; \ell^2)}^2 + E_{\mathcal{N}}(u_s^{n+1}) - E_{\mathcal{N}}(u_s^{n}), \label{sETDMs2: est total}
\end{align}
from which we can obtain the following restraints on $\lambda$ and $A$:
\begin{align}
1-\frac{2\sqrt{3}+1}{4\sqrt{3}\lambda} \geq \frac{1}{4\sqrt{3}\lambda} &\Rightarrow \lambda \geq \frac{3+\sqrt{3}}{6},\label{lambda bd}\\
A-\frac{(2\sqrt{3}+1)\lambda}{4\sqrt{3}} \geq \frac{\lambda}{4\sqrt{3}} &\Rightarrow A \geq \frac{3+\sqrt{3}}{6}\lambda\geq \left(\frac{3+\sqrt{3}}{6}\right)^2 = \frac{2+\sqrt{3}}{6}.\label{A bd}
\end{align}
Setting $\lambda = \frac{3+\sqrt{3}}{6}$, \eqref{sETDMs2: est total} is equal to
\begin{align}
& -\left(A-\frac{2+\sqrt{3}}{6} \right)\tau^2\left\|\frac{\mbox{d} \Delta_{\mathcal{N}} u_s^{n+1}(t)}{\mbox{d} t}\right\|_{L^2(t_n,t_{n+1}; \ell^2)}^2\notag\\
\geq ~&\frac{\sqrt{3}-1}{4}\left\|\frac{\mbox{d} u_s^{n+1}(t)}{\mbox{d} t}\right\|_{L^2(t_n,t_{n+1}; \ell^2)}^2 - \frac{\sqrt{3}-1}{4} \left\|\frac{{\rm d} u_s^{n}(t)}{{\rm d} t}\right\|_{L^2(t_{n-1},t_{n}; \ell^2)}^2 \nonumber \\
&+\frac{(\sqrt{3}+1)\tau^2}{24}\left\|\frac{\mbox{d} \Delta_{\mathcal{N}} u_s^{n+1}(t)}{\mbox{d} t}\right\|_{L^2(t_n,t_{n+1}; \ell^2)}^2 \nonumber \\
& -\frac{(\sqrt{3}+1)\tau^2}{24} \left\|\frac{{\rm d} \Delta_{\mathcal{N}} u_s^{n}(t)}{{\rm d} t}\right\|_{L^2(t_{n-1},t_{n}; \ell^2)}^2  + E_{\mathcal{N}}(u_s^{n+1}) - E_{\mathcal{N}}(u_s^{n})  \nonumber \\
\geq ~&\tilde{E}(u_s^{n+1}) - \tilde{E}(u_s^{n}).\label{sETDMs2: est final}
\end{align}
Thus we have proved the energy stability of \eqref{sETDMs2: discrete}.
\end{proof}

\begin{rem}
With a slight modification of the proof, we can obtain the energy decay property for the following numerical energy functional:
\begin{align}
\hat{E}(u_s^{n+1}) := E_{\mathcal{N}}(u_s^{n+1}) + \frac{\tau}{2\sqrt{3}}\left\|\frac{{\rm d}\nabla_{\mathcal{N}} u_s^{n+1}(t)}{{\rm d} t}\right\|_{L^2(t_n,t_{n+1}; \ell^2)}^2.
\end{align}
Applying the interpolation inequality to the LHS of \eqref{sETDMs2: middle3},  substituting \eqref{I nabla} and \eqref{II nabla} into the RHS of \eqref{sETDMs2: middle3}, we see that
\begin{align*}
&2\sqrt{A}\tau\left\|\frac{{\rm d}\nabla_{\mathcal{N}} u_s^{n+1}(t)}{{\rm d} t}\right\|_{L^2(t_n,t_{n+1}; \ell^2)}^2 + E_{\mathcal{N}}(u_s^{n+1}) - E_{\mathcal{N}}(u_s^{n}) \notag \\
\leq ~& \frac{(6+\sqrt{3})\tau}{6}\left\|\frac{{\rm d}\nabla_{\mathcal{N}} u_s^{n+1}(t)}{{\rm d}t}\right\|_{L^2(t_n,t_{n+1}; \ell^2)}^2+\frac{\tau}{2\sqrt{3}} \left\|\frac{{\rm d} \nabla_{\mathcal{N}} u_s^{n}(t)}{{\rm d} t}\right\|_{L^2(t_{n-1},t_{n}; \ell^2)}^2.
\end{align*}
Rearranging the terms, we also have
\begin{align*}
&\left(2\sqrt{A}-\frac{6+\sqrt{3}}{6}\right)\tau\left\|\frac{{\rm d}\nabla_{\mathcal{N}} u_s^{n+1}(t)}{{\rm d} t}\right\|_{L^2(t_n,t_{n+1};\ell^2)}^2 + E_{\mathcal{N}}(u_s^{n+1})  \notag \\
\leq ~& \frac{\tau}{2\sqrt{3}} \left\|\frac{{\rm d} \nabla_{\mathcal{N}} u_s^{n}(t)}{{\rm d} t}\right\|_{L^2(t_{n-1},t_{n};\ell^2)}^2+ E_{\mathcal{N}}(u_s^{n})=\hat{E}(u_s^{n}).
\end{align*}
Therefore,
\begin{align*}
& \hat{E}(u_s^{n}) + 2\left(\sqrt{A}-\frac{3+\sqrt{3}}{6}\right)\tau\left\|\frac{{\rm d}\nabla_{\mathcal{N}} u_s^{n+1}(t)}{{\rm d} t}\right\|_{L^2(t_n,t_{n+1}; \ell^2)}^2\leq \hat{E}(u_s^{n}) ,
\end{align*}
so that the energy estimate $\hat{E}(u_s^{n+1})\leq \hat{E}(u_s^{n})$ is valid under the same assumption of $A$ as in Theorem~\ref{thm: energy stab}, if we notice that $A\geq\left(\frac{3+\sqrt{3}}{6}\right)^2=\frac{2+\sqrt{3}}{6}$. 
\end{rem}

Given the above energy stability and the fact that $u_s^{n+1}(t)\in \mathcal{M}_0^{\mathcal{N}}$, now we are able to derive a uniform in time bound of $\|u_s^{n+1}(t)\|^2_{H^2_h}$.
\begin{prop}\label{lm: H2 stab}
Assume that the initial solution $u(0)$ has $H^2_{h}(\Omega)$-regularity and $A\geq \frac{2+\sqrt{3}}{6}$. 
Then we have global in time bounds for the numerical solution 
\begin{align}
E_{\mathcal{N}}(u_s^{n+1}) \leq E_{\mathcal{N}}(u_s^{0}), \quad \|u_s^{n+1}(t)\|^2_{H^2_h}\leq C_1, \quad\text{for } 0\leq n\leq N_t-1,
\end{align}
where $C_1$ only depends on $\varepsilon$, $\Omega$ and $\|u(0)\|_{H^2_h}$.
\end{prop}
\begin{proof}
Theorem~\ref{thm: energy stab} implies that
\begin{align}
E_{\mathcal{N}}(u_s^{n+1}) \leq \tilde{E}(u_s^{n+1})\leq \tilde{E}(u_s^{1}).
\end{align}
Taking an inner product with $\frac{{\rm d} u_s^1(t)}{{\rm d} t}$ on both sides of \eqref{sETDMs2: discrete-init} and performing a similar analysis as in \eqref{sETDMs2: middle1}-\eqref{sETDMs2: est1} yields that
\begin{align}
&\left\|\frac{{\rm d} u_s^1(t)}{{\rm d} t}\right\|_{L^2(0,\tau; \ell^2)}^2 + A\tau^2 \left\|\frac{{\rm d}\Delta_{\mathcal{N}} u_s^1 (t)}{{\rm d} t}\right\|_{L^2(0,\tau; \ell^2)}^2  + E_{\mathcal{N}}(u_s^1) - E_{\mathcal{N}}(u_s^0) \nonumber \\
\leq ~& \frac{1}{2}\left\|\frac{{\rm d} u_s^1 (t)}{{\rm d} t}\right\|_{L^2(0,\tau; \ell^2)}^2 + \frac{\tau^2}{2}\left\|\frac{{\rm d} \Delta_{\mathcal{N}} u_s^1 (t)}{{\rm d} t}\right\|_{L^2(0,\tau; \ell^2)}^2.
\end{align}
In addition, the fact $A\geq \frac{2+\sqrt{3}}{6}>\frac{1}{2}$ indicates that
 \begin{align}
 \tilde{E}(u_s^{1}) &\leq \frac{1}{2}\left\|\frac{{\rm d} u_s^1 (t)}{{\rm d} t}\right\|_{L^2(0,\tau; \ell^2)}^2 + \frac{2A-1}{2}\tau^2 \left\|\frac{{\rm d}\Delta_{\mathcal{N}} u_s^1 (t)}{{\rm d} t}\right\|_{L^2(0,\tau; \ell^2)}^2 + E_{\mathcal{N}}(u_s^1) \nonumber \\
 &\leq E_{\mathcal{N}}(u_s^0).
 \end{align}
Therefore, we have $E_{\mathcal{N}}(u_s^{n+1}) \leq E_{\mathcal{N}}(u_s^0)$ for any $0\leq n\leq N_t-1$. Now we can apply Lemma 3.2 and Remark 3.3 in \cite[p. 586]{chen2014linear} to derive the desired conclusion.
\end{proof}

Next we provide a finite time uniform $H^3_h$ bound for the numerical solution $u_s^{n+1}(t)$; this subtle estimate will be used in the optimal rate convergence analysis.
\begin{prop}\label{lm: H3 stab}
Assume that the initial solution $u(0)$ has $H^5_h$-regularity and $A\geq \frac{2+\sqrt{3}}{6}$ . 
Then we have the finite time $H^3_h$ bound for the numerical solution 
\begin{align}
\|u_s^{n+1}(t)\|^2_{H^3_h}\leq C_1, \quad\text{for } 0\leq n\leq N_t-1,
\end{align}
where $C_1$ only depends on $\varepsilon$, $\Omega$ and $\|u_0\|_{H^5_h}$.
\end{prop}
\begin{proof}
Taking an inner product with $-\Delta_{\mathcal{N}}^3 u_s^{n+1}(t)$ on both sides of \eqref{sETDMs2: discrete} leads to
\begin{align}
&\frac{1}{2}\frac{{\rm d}\|\nabla_{\mathcal{N}}\Delta_{\mathcal{N}}u_s^{n+1}(t)\|_{\mathcal{N}}^2}{{\rm d}t} + \frac{A\tau^2}{2}\frac{{\rm d}\|\nabla_{\mathcal{N}}\Delta_{\mathcal{N}}^2 u_s^{n+1}(t)\|_{\mathcal{N}}^2}{{\rm d}t} \notag\\
&+ \varepsilon^2\|\nabla_{\mathcal{N}}\Delta_{\mathcal{N}}^2 u_s^{n+1}(t)\|_{\mathcal{N}}^2\nonumber\\
=~& -\frac{t-t_{n-1}}{\tau}\Big( \nabla_{\mathcal{N}}f_{\mathcal{N}}(u_s^n), \nabla_{\mathcal{N}}\Delta_{\mathcal{N}}^2 u_s^{n+1}(t) \Big)_{\mathcal{N}}\nonumber\\
&+\frac{t-t_n}{\tau}\Big( \nabla_{\mathcal{N}}f_{\mathcal{N}}(u_s^{n-1}), \nabla_{\mathcal{N}}\Delta_{\mathcal{N}}^2 u_s^{n+1}(t) \Big)_{\mathcal{N}}.\label{temp}
\end{align}
For any $v\in H_h^2(\Omega)$ with periodic boundary conditions, recall that $\tilde{v} = \mathcal{I}_Nv$ is the continuous extension of $v$. By \eqref{I_N convergence} we have
\begin{align}
\| \nabla_{\mathcal{N}}f_{\mathcal{N}}(v) \|_{\mathcal{N}} &= \| \nabla_{\mathcal{N}}\nabla_{\mathcal{N}}\cdot\beta(v) \|_{\mathcal{N}} = \| \nabla \nabla\cdot \mathcal{I}_N (\beta(\nabla \tilde{v})) \| \leq C\|\beta(\nabla \tilde{v})\|_{H^2}.
\end{align}
With an application of elliptic regularity, we get
\begin{equation}
\|\beta(\nabla \tilde{v})\|_{H^2}\leq C(\| \Delta \beta(\nabla \tilde{v})\| + \|\beta(\nabla \tilde{v})\|), \quad \text{for }\tilde{v} = \tilde{u}_s^n, ~\tilde{u}_s^{n+1}.
\end{equation}
We also notice that $\|\beta(\nabla \tilde{v})\|\leq \|\nabla \tilde{v}\|\leq C\|\Delta \tilde{v}\|$, and
\begin{align}
| \Delta \beta (\nabla\tilde{v})|
&\leq \frac{|\nabla\Delta \tilde{v}|}{1+|\nabla\tilde{v}|^2} + \frac{6|\Delta\tilde{v}|^2|\nabla\tilde{v}|}{(1+|\nabla\tilde{v}|^2)^2} + \frac{2|\nabla\tilde{v}|^2|\nabla\Delta\tilde{v}|}{(1+|\nabla\tilde{v}|^2)^2} + \frac{8|\nabla\tilde{v}|^3|\Delta\tilde{v}|^2}{(1+|\nabla\tilde{v}|^2)^3}\notag\\
&\leq 3|\nabla\Delta \tilde{v}| + 14|\Delta\tilde{v}|^2.
\end{align}
Also, since $\Delta \tilde{v}$ has zero mean because of the periodic boundary conditions, the Poincar$\rm\acute{e}$ inequality $\|\Delta \tilde{v}\| \leq C\|\nabla\Delta \tilde{v}\|$ is available. Combining this estimate with \eqref{Gagliardo-Nirenberg}, we obtain
\begin{align}
\|\beta(\nabla \tilde{v})\|_{H^2}
&\leq C(\|\nabla\Delta \tilde{v}\| + \|\Delta \tilde{v}\|_{L^{4}}^2)\leq C\|\nabla\Delta \tilde{v}\| + C\|\Delta \tilde{v}\| \|\Delta \tilde{v}\|_{H^1}\leq C\|\nabla\Delta \tilde{v}\|,\label{beta H2 est}
\end{align}
in which we have used the $H^2$ bound of $\tilde{v} = \tilde{u}_s^n, ~\tilde{u}_s^{n+1}$. Substituting the above estimates into \eqref{temp} and applying Lemma~\ref{lm: H2 stab}, we see that
\begin{align}
&\frac{1}{2}\frac{{\rm d}\|\nabla_{\mathcal{N}}\Delta_{\mathcal{N}}u_s^{n+1}(t)\|_{\mathcal{N}}^2}{{\rm d}t} + \frac{A\tau^2}{2}\frac{{\rm d}\|\nabla_{\mathcal{N}}\Delta_{\mathcal{N}}^2 u_s^{n+1}(t)\|_{\mathcal{N}}^2}{{\rm d}t} \notag\\
&+ \varepsilon^2\|\nabla_{\mathcal{N}}\Delta_{\mathcal{N}}^2 u_s^{n+1}(t)\|_{\mathcal{N}}^2\nonumber\\
\leq ~& \frac{\varepsilon^2}{2}\|\nabla_{\mathcal{N}}\Delta_{\mathcal{N}}^2 u_s^{n+1}(t)\|_{\mathcal{N}}^2 + C\varepsilon^{-2}(\|\nabla_{\mathcal{N}}\Delta_{\mathcal{N}}u_s^{n}\|_{\mathcal{N}}^2+\|\nabla_{\mathcal{N}}\Delta_{\mathcal{N}}u_s^{n-1}\|_{\mathcal{N}}^2).
\end{align}
In turn, an integration from $t_n$ to $t_{n+1}$ implies that
\begin{align}
&\frac{1}{2}\left(\|\nabla_{\mathcal{N}}\Delta_{\mathcal{N}}u_s^{n+1}\|_{\mathcal{N}}^2 - \|\nabla_{\mathcal{N}}\Delta_{\mathcal{N}}u_s^{n}\|_{\mathcal{N}}^2\right) \notag\\
& \quad + \frac{A\tau^2}{2}\left(\|\nabla_{\mathcal{N}}\Delta_{\mathcal{N}}^2 u_s^{n+1}\|_{\mathcal{N}}^2 - \|\nabla_{\mathcal{N}}\Delta_{\mathcal{N}}^2 u_s^{n}\|_{\mathcal{N}}^2\right)\nonumber\\
\leq ~& C\varepsilon^{-2}\tau(\|\nabla_{\mathcal{N}}\Delta_{\mathcal{N}}u_s^{n}\|_{\mathcal{N}}^2+\|\nabla_{\mathcal{N}}\Delta_{\mathcal{N}}u_s^{n-1}\|_{\mathcal{N}}^2).
\end{align}
A summation from $n=1$ to $n$ shows that
\begin{align*}
&\frac{1}{2}\|\nabla_{\mathcal{N}}\Delta_{\mathcal{N}}u_s^{n+1}\|_{\mathcal{N}}^2 + \frac{A\tau^2}{2}\|\nabla_{\mathcal{N}}\Delta_{\mathcal{N}}^2 u_s^{n+1}\|_{\mathcal{N}}^2\nonumber\\
\leq ~& C\tau\varepsilon^{-2}\sum_{i=0}^n \|\nabla_{\mathcal{N}}\Delta_{\mathcal{N}}u_s^{i}\|_{\mathcal{N}}^2 + \frac{1}{2}\|\nabla_{\mathcal{N}}\Delta_{\mathcal{N}}u_s^{0}\|_{\mathcal{N}}^2 + \frac{A\tau^2}{2}\|\nabla_{\mathcal{N}}\Delta_{\mathcal{N}}^2 u_s^{0}\|_{\mathcal{N}}^2.
\end{align*}
An application pf the discrete Gronwall inequality yields
\begin{align}
\|\nabla_{\mathcal{N}}\Delta_{\mathcal{N}}u_s^{n+1}\|_{\mathcal{N}} \leq C,
\end{align}
where $C$ depends only on $\varepsilon, ~u(0)$ and $T$.
\end{proof}

\section{Error analysis of the sETDMs2 scheme}\label{sec: error analysis}
The goal of this section is to provide an optimal rate convergence analysis for the sETDMs2 scheme \eqref{sETDMs2}. Recall that $h=\frac{L}{2N}$, $\tau = \frac{T}{N_t}$, $\beta(\mathbf{v}) = \frac{\mathbf{v}}{1+|\mathbf{v}|^2}$.

The following inverse estimate is needed in the error analysis.
\begin{lem}\label{inverse estimate-2}
For any $u\in \mathcal{B}^{N}$ and $2\leq p\leq\infty$,
\begin{align}
|u|_{W^{m,p}(\Omega)}&\leq Ch^{-k-\left(\frac{1}{2}-\frac{1}{p}\right)d}\|u\|_{H^{m-k}(\Omega)}, ~0\leq k\leq m.
\end{align}
\end{lem}
\begin{proof}
By \eqref{Gagliardo-Nirenberg} and \eqref{Agmon} we have
\begin{align*}
\|\partial_x^m u\|_{L^p(\Omega)}\leq \|\partial_x^m u\|_{L^2(\Omega)}^{1-\theta}\|\partial_x^m u\|_{H^{k_0}(\Omega)}^{\theta}, ~\theta = \left(\frac{1}{2}-\frac{1}{p}\right)\frac{d}{k_0},
\end{align*}
where the integer ${k_0}\ge (\frac{1}{2}-\frac{1}{p})d$ for $2\le p<\infty$ and $k_0>\frac{d}{2}$ for $p=\infty$. Hence,
\begin{align}
|u|_{W^{m,p}(\Omega)}\leq |u|_{H^{m}(\Omega)}^{1-\theta}\|u\|_{H^{m+k_0}(\Omega)}^{\theta}.\label{inv mid-1}
\end{align}
Meanwhile, an application of~\eqref{inverse estimate-0} (in Lemma~\ref{inverse estimate}) implies that
\begin{align}
|u|_{H^{m+k_0}(\Omega)}\leq Ch^{-k_0}|u|_{H^{m}(\Omega)}, ~k_0\geq 1.\label{inv mid-2}
\end{align}
Combining \eqref{inv mid-1} and \eqref{inv mid-2} leads to
\begin{align}
|u|_{W^{m,p}(\Omega)}\leq Ch^{-\left(\frac{1}{2}-\frac{1}{p}\right)d}\|u\|_{H^{m}(\Omega)}.\label{inv mid-3}
\end{align}
For $k\ge 1$, we simply apply \eqref{inverse estimate} again:
\begin{align}
|u|_{W^{m,p}(\Omega)}\leq Ch^{-k-\left(\frac{1}{2}-\frac{1}{p}\right)d}\|u\|_{H^{m-k}(\Omega)}. \label{inv mid-4}
\end{align}
We get the conclusion.
\end{proof}

The following estimate is needed in the analysis of the nonlinear term.
\begin{lem}\label{lm: sETDMs2 nonlinear - middle step}
For any $v, ~w\in\mathcal{M}_0^{\mathcal{N}}\cap H^3_h, ~g\in\mathcal{M}^{\mathcal{N}}\cap H^2_h$, we have
\begin{align*}
&\left\|\nabla_{\mathcal{N}}\cdot\left(\frac{\nabla_{\mathcal{N}}g}{1+|\nabla_{\mathcal{N}} v|^2}\right)\right\|_{\mathcal{N}}
\leq C(1+h)(\|\Delta_{\mathcal{N}}g\|_{\mathcal{N}} + \|g\|_{\mathcal{N}}),\\
&\left\| \nabla_{\mathcal{N}} \cdot \left(\frac{\nabla_{\mathcal{N}} w \cdot \nabla_{\mathcal{N}}g (\nabla_{\mathcal{N}} w + \nabla_{\mathcal{N}} v)}{(1+|\nabla_{\mathcal{N}} v|^2)(1+|\nabla_{\mathcal{N}} w|^2)}\right) \right\|_{\mathcal{N}}
\leq C(1+h)(\|\Delta_{\mathcal{N}}g\|_{\mathcal{N}} + \|g\|_{\mathcal{N}}),
\end{align*}
in which $C$ depends only on $\Omega$, $\|v\|_{H^3_h}$ and $\|w\|_{H^3_h}$.
\end{lem}
\begin{proof}
Define $\tilde{v} = \mathcal{I}_Nv, ~\tilde{w} = \mathcal{I}_Nw, ~\tilde{g} = \mathcal{I}_Ng$ as the continuous extension of $v, ~w, ~g$, respectively. Thanks to \eqref{I_N convergence}, we have
\begin{align}
\left\|\nabla_{\mathcal{N}}\cdot\left(\frac{\nabla_{\mathcal{N}}g}{1+|\nabla_{\mathcal{N}} v|^2}\right)\right\|_{\mathcal{N}}
&= \left\|\nabla\cdot\mathcal{I}_N\left(\frac{\nabla \tilde{g}}{1+|\nabla \tilde{v}|^2}\right)\right\| \nonumber\\
&\leq \left\|\nabla\cdot\left(\frac{\nabla \tilde{g}}{1+|\nabla \tilde{v}|^2}\right)\right\| + Ch\left\|\frac{\nabla \tilde{g}}{1+|\nabla \tilde{v}|^2}\right\|_{H^2}.\label{temp2}
\end{align}
For the first term on the RHS, using the Gagliardo-Nirenberg inequality \eqref{Gagliardo-Nirenberg}, the interpolation inequality, the $H^3$ bound of $\|\tilde{v}\|$ and the Young inequality, we get
\begin{align}
\left\|\nabla\cdot\left(\frac{\nabla \tilde{g}}{1+|\nabla \tilde{v}|^2}\right)\right\|
&\leq C(|\tilde{g}|_{H^2} + |\tilde{g}|_{1,4}|\tilde{v}|_{2,4}) \leq C|\tilde{g}|_{H^2} + C\|\tilde{g}\|_{H^1}^{\frac{1}{2}}\|\tilde{g}\|_{H^2}^{\frac{1}{2}}\|\tilde{v}\|_{H^3}\nonumber\\
&\leq C|\tilde{g}|_{H^2} + C\|\tilde{g}\|_{H^2}\leq C\|\tilde{g}\| + C\|\Delta\tilde{g}\|.\label{44}
\end{align}
For the second term, the elliptic regularity and the interpolation inequality shows that
\begin{align*}
\left\|\frac{\nabla \tilde{g}}{1+|\nabla \tilde{v}|^2}\right\|_{H^2}
&\leq C\left(\left\| \Delta\left(\frac{\nabla \tilde{g}}{1+|\nabla \tilde{v}|^2}\right) \right\|  + \left\|\frac{\nabla \tilde{g}}{1+|\nabla \tilde{v}|^2}\right\| \right)\\
&\leq C\left(\left\| \Delta\left(\frac{\nabla \tilde{g}}{1+|\nabla \tilde{v}|^2}\right) \right\|  + \|\tilde{g}\| +  \|\Delta\tilde{g}\| \right).
\end{align*}
On the other hand, we see that
\begin{align}
\left| \Delta\left(\frac{\nabla \tilde{g}}{1+|\nabla \tilde{v}|^2}\right) \right|
&\leq C(|\nabla\Delta\tilde{g}| + |\nabla\nabla\tilde{g}||\nabla\nabla\tilde{v}| + |\nabla\tilde{g}| |\nabla\nabla\tilde{v}|^2 + |\nabla\tilde{g}||\Delta\nabla\tilde{v}|).
\end{align}
Then we obtain
\begin{align}
\left\| \Delta\left(\frac{\nabla \tilde{g}}{1+|\nabla \tilde{v}|^2}\right) \right\|
&\leq C(|\tilde{g}|_{H^{3}} + |\tilde{g}|_{2,4}|\tilde{v}|_{2,4} + |\tilde{g}|_{1,6}|\tilde{v}|_{2,6}^2  + |\tilde{g}|_{1,\infty} |\tilde{v}|_{H^{3}}).
\end{align}
Applying Lemma~\ref{inverse estimate}, Lemma~\ref{inverse estimate-2} and the Sobolev inequality, we get
\begin{align}
&~\left\| \Delta\left(\frac{\nabla \tilde{g}}{1+|\nabla \tilde{v}|^2}\right) \right\| \nonumber\\
\leq ~& Ch^{-1}|\tilde{g}|_{H^{2}} + Ch^{-\frac{1}{2}}\|\tilde{g}\|_{H^2}\|\tilde{v}\|_{H^3}+ C\|\tilde{g}\|_{H^2}\|\tilde{v}\|_{H^{3}}^{2} + Ch^{-1}|\tilde{g}|_{L^{\infty}}|\tilde{v}|_{H^{3}} \nonumber\\
\leq ~& Ch^{-1}|\tilde{g}|_{H^{2}} + Ch^{-\frac{1}{2}}\|\tilde{g}\|_{H^2}\|\tilde{v}\|_{H^3} + C\|\tilde{g}\|_{H^{2}}\|\tilde{v}\|_{H^{3}}^{2} + Ch^{-1}\|\tilde{g}\|_{H^{2}}\|\tilde{v}\|_{H^3} \nonumber\\
\leq ~& C(h^{-1}+1)\|\tilde{g}\|_{H^{2}}.\label{complex est}
\end{align}
Substituting the above estimate into \eqref{temp2}, we arrive at
\begin{align}
\left\|\nabla_{\mathcal{N}}\cdot\left(\frac{\nabla_{\mathcal{N}}g}{1+|\nabla_{\mathcal{N}} v|^2}\right)\right\|_{\mathcal{N}}
\leq~& C\|\Delta\tilde{g}\| + C\|\tilde{g}\| + C(1+h)\|\tilde{g}\|_{H^{2}}\nonumber\\
\leq ~& C(1+h)\|\Delta \tilde{g}\| + C(1+h)\|\tilde{g}\|\notag\\
= ~& C(1+h)(\|\Delta_{\mathcal{N}}g\|_{\mathcal{N}} + \|g\|_{\mathcal{N}}).
\end{align}
Similar estimates can be applied on the second part of the conclusion:
\begin{align}
&\left\| \nabla_{\mathcal{N}} \cdot \left(\frac{\nabla_{\mathcal{N}} w \cdot \nabla_{\mathcal{N}}g (\nabla_{\mathcal{N}} w + \nabla_{\mathcal{N}} v)}{(1+|\nabla_{\mathcal{N}} v|^2)(1+|\nabla_{\mathcal{N}} w|^2)}\right) \right\|_{\mathcal{N}} \nonumber\\
=~& \left\| \nabla \cdot \mathcal{I}_N \left(\frac{\nabla \tilde{w} \cdot \nabla g (\nabla \tilde{w} + \nabla \tilde{v})}{(1+|\nabla \tilde{v}|^2)(1+|\nabla \tilde{w}|^2)}\right) \right\| \nonumber\\
\leq~& \left\| \nabla \cdot \left(\frac{\nabla \tilde{w} \cdot \nabla g (\nabla \tilde{w} + \nabla \tilde{v})}{(1+|\nabla \tilde{v}|^2)(1+|\nabla \tilde{w}|^2)}\right) \right\| + Ch \left\| \frac{\nabla \tilde{w} \cdot \nabla g (\nabla \tilde{w} + \nabla \tilde{v})}{(1+|\nabla \tilde{v}|^2)(1+|\nabla \tilde{w}|^2)}\right\|_{H^2}.\label{temp3}
\end{align}
Analyzing the first term on the RHS in a similar fashion as in~\eqref{44}, we have
\begin{align}
& \left\| \nabla \cdot \left(\frac{\nabla \tilde{w} \cdot \nabla g (\nabla \tilde{w} + \nabla \tilde{v})}{(1+|\nabla \tilde{v}|^2)(1+|\nabla \tilde{w}|^2)}\right) \right\| \nonumber\\
\leq~& C(|\tilde{w}|_{2,4}|\tilde{g}|_{1,4} + |\tilde{g}|_{H^2} + |\tilde{g}|_{1,4}(|\tilde{w}|_{2,4}+|\tilde{v}|_{2,4}) )
\leq C\|\Delta\tilde{g}\|   + C\|\tilde{g}\|.\label{51}
\end{align}
The second term can also be analyzed by the elliptic regularity and the interpolation inequality:
\begin{align*}
&\left\| \frac{\nabla \tilde{w} \cdot \nabla g (\nabla \tilde{w} + \nabla \tilde{v})}{(1+|\nabla \tilde{v}|^2)(1+|\nabla \tilde{w}|^2)}\right\|_{H^2} \\
\leq ~&C\left\| \Delta\left(\frac{\nabla \tilde{w} \cdot \nabla g (\nabla \tilde{w} + \nabla \tilde{v})}{(1+|\nabla \tilde{v}|^2)(1+|\nabla \tilde{w}|^2)}\right)\right\| + C\left\| \frac{\nabla \tilde{w} \cdot \nabla g (\nabla \tilde{w} + \nabla \tilde{v})}{(1+|\nabla \tilde{v}|^2)(1+|\nabla \tilde{w}|^2)}\right\| \\
\leq ~& C\left\| \Delta\left(\frac{\nabla \tilde{w} \cdot \nabla g (\nabla \tilde{w} + \nabla \tilde{v})}{(1+|\nabla \tilde{v}|^2)(1+|\nabla \tilde{w}|^2)}\right)\right\| + C(\|\tilde{g}\| +  |\tilde{g}|_{H^2}).
\end{align*}
Also notice that
\begin{align}
&\left\| \Delta\left(\frac{\nabla \tilde{w} \cdot \nabla g (\nabla \tilde{w} + \nabla \tilde{v})}{(1+|\nabla \tilde{v}|^2)(1+|\nabla \tilde{w}|^2)}\right)\right\| \nonumber\\
\leq~& C|\tilde{w}|_{H^3}|\tilde{g}|_{1,\infty} + C|\tilde{g}|_{H^3} + C|\tilde{g}|_{1,4}(|\tilde{w}|_{3,4} + C|\tilde{v}|_{3,4}) + C|\tilde{g}|_{1,6}|\tilde{w}|_{2,6}^2 \nonumber\\
& + C|\tilde{g}|_{1,6}|\tilde{v}|_{2,6}^2 + C|\tilde{g}|_{1,6}|\tilde{w}|_{2,6}|\tilde{v}|_{2,6} + C|\tilde{g}|_{2,4}\left( |\tilde{w}|_{2,4} + |\tilde{v}|_{2,4} \right).
\end{align}
Again we apply Lemma~\ref{inverse estimate}, \eqref{Gagliardo-Nirenberg}, \eqref{Agmon} and the Sobolev inequality as in \eqref{complex est}:
\begin{align}
&\left\| \frac{\nabla \tilde{w} \cdot \nabla g (\nabla \tilde{w} + \nabla \tilde{v})}{(1+|\nabla \tilde{v}|^2)(1+|\nabla \tilde{w}|^2)}\right\|_{H^2} \nonumber\\
\leq ~& C(\|\tilde{g}\| +  |\tilde{g}|_{H^2}) + Ch^{-1}\|\tilde{w}\|_{H^3}\|\tilde{g}\|_{L^{\infty}} + Ch^{-1}\|\tilde{g}\|_{H^2}  \nonumber\\
& + Ch^{-1}\|\tilde{g}\|_{H^{1}}(\|\tilde{w}\|_{H^{3}} + \|\tilde{v}\|_{H^{3}}) \notag\\
& + C\|\tilde{g}\|_{H^2}(\|\tilde{w}\|_{H^3}^{2} + \|\tilde{v}\|_{H^3}^{2} +\|\tilde{w}\|_{H^3}\|\tilde{v}\|_{H^3})\notag\\
& + Ch^{-1}\|\tilde{g}\|_{H^2}\left( \|\tilde{w}\|_{H^2} + \|\tilde{v}\|_{H^2} \right) \nonumber\\
\leq ~& C(1+h^{-1})\|\tilde{g}\|_{H^2}\leq C(1+h^{-1})(\|\tilde{g}\| + \|\Delta\tilde{g}\|).\label{53}
\end{align}
Substituting \eqref{51} and \eqref{53} into \eqref{temp3} yields that
\begin{align}
&\left\| \nabla_{\mathcal{N}} \cdot \left(\frac{\nabla_{\mathcal{N}} w \cdot \nabla_{\mathcal{N}}g (\nabla_{\mathcal{N}} w + \nabla_{\mathcal{N}} v)}{(1+|\nabla_{\mathcal{N}} v|^2)(1+|\nabla_{\mathcal{N}} w|^2)}\right) \right\|_{\mathcal{N}}\nonumber\\
\leq ~& C(1+h)\|\Delta\tilde{g}\| + C(1+h)\|\tilde{g}\| = C(1+h)(\|\Delta_{\mathcal{N}} g\|_{\mathcal{N}} + \|g\|_{\mathcal{N}}),
\end{align}
which is the desired conclusion.
\end{proof}

Now we present an estimate of the nonlinear term.
\begin{lem}\label{lm: sETDMs2 nonlinear}
For any $v, ~w, ~\mathcal{M}_0^{\mathcal{N}}\cap H^3_h$ and $g\in\mathcal{M}^{\mathcal{N}}\cap H^2_h$, we have
\begin{align*}
&\Big(\beta(\nabla_{\mathcal{N}} v) - \beta(\nabla_{\mathcal{N}} w) , \nabla_{\mathcal{N}}g\Big)_{\mathcal{N}} \notag\\
\leq ~&\frac{C_{v,w}(2+\varepsilon^2)(1+h)}{2\varepsilon^2}\|v - w\|_{\mathcal{N}}^2 + \frac{\varepsilon^2}{4}\|\Delta_{\mathcal{N}}g\|_{\mathcal{N}}^2+ \frac{1}{2}\|g\|_{\mathcal{N}}^2,
\end{align*}
where $C_{v,w}$ is a constant depending on $\Omega$, $\|w\|_{H^3_h}$ and $\|v\|_{H^3_h}$.
\end{lem}
\begin{proof}
We begin with the following identity: For any $v, ~w\in\mathcal{M}_0^{\mathcal{N}}$,
\begin{align}
&\Big(\beta(\nabla_{\mathcal{N}} v) - \beta(\nabla_{\mathcal{N}} w) , \nabla_{\mathcal{N}}g\Big)_{\mathcal{N}} \nonumber \\
=& \left( \beta(\nabla_{\mathcal{N}} v) - \frac{\nabla_{\mathcal{N}} w}{1+|\nabla_{\mathcal{N}} v|^2}, \nabla_{\mathcal{N}}g \right)_{\mathcal{N}} + \left( \frac{\nabla_{\mathcal{N}} w}{1+|\nabla_{\mathcal{N}} v|^2} - \beta(\nabla_{\mathcal{N}} w), \nabla_{\mathcal{N}}g \right)_{\mathcal{N}}\nonumber \\
=& \left(\nabla_{\mathcal{N}} (v - w), \frac{\nabla_{\mathcal{N}}g}{1+|\nabla_{\mathcal{N}} v|^2} \right)_{\mathcal{N}} + \left( \nabla_{\mathcal{N}} w \frac{|\nabla_{\mathcal{N}} w|^2 - |\nabla_{\mathcal{N}} v|^2}{(1+|\nabla_{\mathcal{N}} v|^2)(1+|\nabla_{\mathcal{N}} w|^2)}, \nabla_{\mathcal{N}}g \right)_{\mathcal{N}}\nonumber \\
=& \left(\nabla_{\mathcal{N}} (v - w), \frac{\nabla_{\mathcal{N}}g}{1+|\nabla_{\mathcal{N}} v|^2} \right)_{\mathcal{N}} + \left( \nabla_{\mathcal{N}} (w - v), \frac{\nabla_{\mathcal{N}} w \cdot \nabla_{\mathcal{N}}g (\nabla_{\mathcal{N}} w + \nabla_{\mathcal{N}} v)}{(1+|\nabla_{\mathcal{N}} v|^2)(1+|\nabla_{\mathcal{N}} w|^2)} \right)_{\mathcal{N}}.\nonumber
\end{align}
An application of integration by parts gives
\begin{align}
&\left(\frac{\nabla_{\mathcal{N}} v}{1+|\nabla_{\mathcal{N}} v|^2} - \frac{\nabla_{\mathcal{N}} w}{1+|\nabla_{\mathcal{N}} w|^2}, \nabla_{\mathcal{N}}g\right)_{\mathcal{N}} \nonumber \\
=~& -\left(v - w, \nabla_{\mathcal{N}}\cdot\left(\frac{\nabla_{\mathcal{N}}g}{1+|\nabla_{\mathcal{N}} v|^2}\right) \right)_{\mathcal{N}} \nonumber \\
& - \left( w - v, \nabla_{\mathcal{N}} \cdot \left(\frac{\nabla_{\mathcal{N}} w \cdot \nabla_{\mathcal{N}}g (\nabla_{\mathcal{N}} w + \nabla_{\mathcal{N}} v)}{(1+|\nabla_{\mathcal{N}} v|^2)(1+|\nabla_{\mathcal{N}} w|^2)}\right) \right)_{\mathcal{N}} \nonumber \\
\leq ~& \|v - w\|_{\mathcal{N}}\left\|\nabla_{\mathcal{N}}\cdot\left(\frac{\nabla_{\mathcal{N}}g}{1+|\nabla_{\mathcal{N}} v|^2}\right)\right\|_{\mathcal{N}} \nonumber \\
& + \|v - w\|_{\mathcal{N}}\left\| \nabla_{\mathcal{N}} \cdot \left(\frac{\nabla_{\mathcal{N}} w \cdot \nabla_{\mathcal{N}}g (\nabla_{\mathcal{N}} w + \nabla_{\mathcal{N}} v)}{(1+|\nabla_{\mathcal{N}} v|^2)(1+|\nabla_{\mathcal{N}} w|^2)}\right) \right\|_{\mathcal{N}}.\label{sETDMs2: NL part}
\end{align}
Applying Lemma~\ref{lm: sETDMs2 nonlinear - middle step}, we have the following estimate:
\begin{align}
&\left(\frac{\nabla_{\mathcal{N}} v}{1+|\nabla_{\mathcal{N}} v|^2} - \frac{\nabla_{\mathcal{N}} w}{1+|\nabla_{\mathcal{N}} w|^2}, \nabla_{\mathcal{N}}g\right)_{\mathcal{N}}  \nonumber \\
\leq ~& C(1+h)\|v - w\|_{\mathcal{N}}\left(\|\Delta_{\mathcal{N}} g\|_{\mathcal{N}} + \|g\|_{\mathcal{N}}\right) \nonumber \\
\leq ~& \frac{\varepsilon^2}{4}\|\Delta_{\mathcal{N}}g\|_{\mathcal{N}}^2 + \frac{1}{2}\|g\|_{\mathcal{N}}^2+ C(1+h)\left(\frac{1}{\varepsilon^2} + \frac{1}{2}\right)\|v - w\|_{\mathcal{N}}^2,\label{nonlinear term 1}
\end{align}
where $C$ is a constant dependent on $\|w\|_{H^3_h}$, $\|v\|_{H^3_h}$, $\Omega_{\mathcal{N}}$.
\end{proof}

Below is the main result of this section.
\begin{thm}\label{thm:convergence}
Assume that the exact solution satisfies $u_e\in H^{1}(0,T;H^{m+4}_{per}(\Omega))\cap H^{2}(0,T;H^2(\Omega))$ with $m\geq 0$. Let $u(t) = u_e(t)|_{\Omega_{\mathcal{N}}}$ and $u_s^0 = u(0)\in H_h^5$. Denote by $\{u_s^n\}^{N_t}_{n=1}$ the solution of the sETDMs2 scheme \eqref{sETDMs2}. If the time step satisfies
\begin{equation}
\tau \leq \frac{1}{8},
\end{equation}
then we have
\begin{equation}
\|u(t_n)-u_s^n\|_{\mathcal{N}}\leq C(\tau^2+N^{-m}), \quad 1\leq n\leq N_t,
\end{equation}
in which $C>0$ is independent of $\tau$ and the spatial discretization parameter $N$.
\end{thm}
\begin{proof}
Define error function $e(t) = u(t) - u_s^{n+1}(t)\in \mathcal{M}^{\mathcal{N}}$. Below we denote $u_e(t_n)$, $u(t_n)$ as $u_e^n$ and $u^n$, respectively. Subtracting \eqref{sETDMs2: discrete} from \eqref{MBE} gives
\begin{align}
&\frac{{\rm d} e(t)}{{\rm d} t} + A\tau^2\frac{{\rm d}\Delta^2_{\mathcal{N}} e(t)}{{\rm d} t} + \varepsilon^2\Delta_{\mathcal{N}}^2 e(t)\nonumber \\
=~& - \left[f_{\mathcal{N},L}(t - t_n, u^n,u^{n-1}) - f_{\mathcal{N},L}(t - t_n, u_s^{n}, u_s^{n-1})\right] + \mathcal{R}(t),\label{sETDMs2: error eq}
\end{align}
for $t\leq t_{n+1}$, where $\mathcal{R}(t)$ is the truncation error:
\begin{align}
\mathcal{R}(t)
&= -\varepsilon^2 (\Delta^2 - \Delta^2_{\mathcal{N}})u(t) + A\tau^2\frac{{\rm d}\Delta^2_{\mathcal{N}} u(t)}{{\rm d} t} \notag\\
&\quad + \Big[\nabla_{\mathcal{N}}\cdot\beta(\nabla_{\mathcal{N}} u(t)) - \nabla\cdot\beta(\nabla u_e(t))\Big|_{\Omega_{\mathcal{N}}}\Big] \nonumber\\
&\quad + \Big(\frac{t-t_n}{\tau}\nabla_{\mathcal{N}}\cdot \left[\beta(\nabla_{\mathcal{N}}u(t)) - \beta(\nabla_{\mathcal{N}}u^{n-1})\right]\nonumber\\
&\qquad \quad -\frac{t-t_{n-1}}{\tau}\nabla_{\mathcal{N}}\cdot \left[\beta(\nabla_{\mathcal{N}} u(t)) - \beta(\nabla_{\mathcal{N}}u^n)\right]\Big)\nonumber \\
&:= \mathcal{R}_1(t) + \mathcal{R}_2(t) + \mathcal{R}_3(t) + \mathcal{R}_4(t).\label{sETDMs2: truncation}
\end{align}
Taking an inner product with $e(t)$ on both sides of \eqref{sETDMs2: truncation} yields
\begin{align}
&\frac{1}{2}\frac{{\rm d} \|e(t)\|_{\mathcal{N}}^2}{{\rm d} t} + \frac{A\tau^2}{2}\frac{{\rm d}\|\Delta_{\mathcal{N}} e(t)\|_{\mathcal{N}}^2}{{\rm d} t} + \varepsilon^2\|\Delta_{\mathcal{N}} e(t)\|_{\mathcal{N}}^2 \nonumber \\
=~& \left(\mathbf{\hat{f}}_{\mathcal{N},L}(t - t_n, u(t_n),u(t_{n-1})) - \mathbf{\hat{f}}_{\mathcal{N},L}(t - t_n, u_s^{n},u_s^{n-1}), \nabla_{\mathcal{N}}e(t) \right)_{\mathcal{N}}\notag\\
& + \left(\mathcal{R}(t), e(t)_{\mathcal{N}}\right)\nonumber \\
=~& \text{(NL)} + \left(\mathcal{R}(t), e(t)\right)_{\mathcal{N}}.\label{sETDMs2: error eq 2}
\end{align}
The truncation error term could be handled by the Cauchy-Schwarz inequality:
\begin{align}
\left(\mathcal{R}(t), e(t)\right)_{\mathcal{N}} &\leq \frac{\|\mathcal{R}(t)\|_{\mathcal{N}}^2}{2}+\frac{\|e(t)\|_{\mathcal{N}}^2}{2}.\label{sETDMs2: truncation error}
\end{align}
The nonlinear term on the RHS has the following form:
\begin{align}
\text{(NL)} &= \frac{t + \tau - t_n}{\tau}\Big(\beta(\nabla_{\mathcal{N}} u^n) - \beta(\nabla_{\mathcal{N}} u_s^n), \nabla_{\mathcal{N}}e(t)\Big)_{\mathcal{N}} \nonumber \\
&\quad \frac{-t+t_n}{\tau}\Big(\beta(\nabla_{\mathcal{N}} u^{n-1}) - \beta(\nabla_{\mathcal{N}} u_s^{n-1}), \nabla_{\mathcal{N}}e(t)\Big)_{\mathcal{N}}.\label{sETDMs2: NL}
\end{align}
An application of Lemma~\ref{lm: H3 stab} and Lemma~\ref{lm: sETDMs2 nonlinear} to \eqref{sETDMs2: NL} results in
\begin{align}
\text{(NL)} &\leq \frac{C(1+h)(2+\varepsilon^2)}{2\varepsilon^2}\left(\|e(t_n)\|_{\mathcal{N}}^2 + \|e(t_{n-1})\|_{\mathcal{N}}^2\right) + \frac{3\varepsilon^2}{4}\|\Delta_{\mathcal{N}}e(t)\|_{\mathcal{N}}^2 + \frac{3}{2}\|e(t)\|_{\mathcal{N}}^2,\label{sETDMs2: NL-final}
\end{align}
where $C$ depends on $\|e(t_n)\|_{H^3_h}$, $\|e(t_{n-1})\|_{H^3_h}$ and $\Omega_{\mathcal{N}}$, and thus is a constant dependent only of $\|u_0\|_{H^5_h}$ and $\Omega_{\mathcal{N}}$.

Substituting \eqref{sETDMs2: truncation error} and \eqref{sETDMs2: NL-final} into \eqref{sETDMs2: error eq 2}, we obtain
\begin{align}
&\frac{1}{2}\frac{{\rm d} \|e(t)\|_{\mathcal{N}}^2}{{\rm d} t} + \frac{A\tau^2}{2}\frac{{\rm d}\|\Delta_{\mathcal{N}} e(t)\|_{\mathcal{N}}^2}{{\rm d} t} + \frac{\varepsilon^2}{4} \|\Delta_{\mathcal{N}} e(t)\|_{\mathcal{N}}^2 \nonumber \\
\leq~& \frac{C(1+h)(2+\varepsilon^2)}{2\varepsilon^2}\left(\|e(t_n)\|_{\mathcal{N}}^2 + \|e(t_{n-1})\|_{\mathcal{N}}^2\right) + 2\|e(t)\|_{\mathcal{N}}^2 + \frac{\|\mathcal{R}(t)\|_{\mathcal{N}}^2}{2}.\label{sETDMs2: error eq 4}
\end{align}
Denote $\omega (t) = \frac{1}{2}\|e(t)\|_{\mathcal{N}}^2 + \frac{A\tau^2}{2}\|\Delta_{\mathcal{N}} e(t)\|_{\mathcal{N}}^2$. Multiplying both sides by $e^{-4t}$ gives
\begin{align}
\frac{{\rm d}}{{\rm d}t}e^{-4t}\omega(t) &\leq \left[  \frac{C(1+h)(2+\varepsilon^2)}{2\varepsilon^2}\left(\|e(t_n)\|_{\mathcal{N}}^2 + \|e(t_{n-1})\|_{\mathcal{N}}^2\right) + \frac{\|\mathcal{R}(t)\|_{\mathcal{N}}^2}{2} \right]e^{-4t}.\label{sETDMs2: error eq 3}
\end{align}
Integrating \eqref{sETDMs2: error eq 3} from $t_n$ to $t_{n+1}$ and multiplying both sides by $e^{4t_{n}}$, we have
\begin{align}
&e^{-4\tau}\omega (t_{n+1}) - \omega (t_n) \nonumber \\
\leq & \frac{1-e^{-4\tau}}{4}\frac{C(1+h)(2+\varepsilon^2)}{\varepsilon^2}\left(\|e(t_n)\|_{\mathcal{N}}^2 + \|e(t_{n-1})\|_{\mathcal{N}}^2\right)+ \|\mathcal{R}(t)\|_{L^2(t_n,t_{n+1}; \ell^2)}^2. \nonumber
\end{align}
Since $e^x\geq 1+x$ for $x\in\mathbb{R}$, from the above inequality we can further obtain that
\begin{align}
&\omega (t_{n+1}) - \omega (t_n) - 4\tau\omega (t_{n+1})  \nonumber \\
\leq ~& \frac{C(1+h)(2+\varepsilon^2)\tau}{\varepsilon^2}\left(\|e(t_n)\|_{\mathcal{N}}^2 + \|e(t_{n-1})\|_{\mathcal{N}}^2\right)+ \sum_{i=1}^4 \|R_i(t)\|_{L^2(t_n,t_{n+1}; \ell^2)}^2. \label{gronwall-mid}
\end{align}
Note that $\mathcal{R}_i(t) = \widetilde{\mathcal{R}}_i(t)\Big|_{\Omega_{\mathcal{N}}}$, for $1\leq i \leq 3$, where
\begin{align}
\|\widetilde{\mathcal{R}}_1(t)\|_{L^2} &= \|-\varepsilon^2 \left(\Delta^2 u_e(t) - \Delta^2 \mathcal{I}_Nu(t)\right)\|_{L^2}\leq C\varepsilon^2h^{m}\|u_e(t)\|_{H^{m+4}},\\
\|\widetilde{\mathcal{R}}_2(t)\|_{L^2} &= \left\|A\tau^2\frac{{\rm d}\Delta^2 \mathcal{I}_N u_e(t)}{{\rm d} t}\right\|_{L^2}= A\tau^2\left\|\Delta^2\mathcal{I}_N\frac{{\rm d} u_e(t)}{{\rm d} t}\right\|_{L^2} \nonumber\\
																&\leq C\tau^2 h^{m}\left\|\frac{{\rm d} u_e(t)}{{\rm d} t}\right\|_{H^{m+4}} + A\tau^2\left\|\frac{{\rm d} u_e(t)}{{\rm d} t}\right\|_{H^4},\\
\|\widetilde{\mathcal{R}}_3(t)\|_{L^2} &= \|\nabla\cdot\left[\beta(\nabla_{\mathcal{N}} \mathcal{I}_Nu_e(t)) - \beta(\nabla u_e(t))\right]\|_{L^2}\leq C\|\mathcal{I}_Nu_e(t) - u_e(t)\|_{H^2}\nonumber\\																
&\leq C\|u_e(t)\|_{H^{m+2}}h^{m},
\end{align}
where we have applied the approximation property of $\mathcal{I}_N$ in \eqref{I_N convergence}. As for $\mathcal{R}_4(t)$, let $\alpha(t) := \nabla_{\mathcal{N}}\cdot \beta(\nabla_{\mathcal{N}} u(t))$, so that
\begin{align}
\mathcal{R}_4(t)
&= -\alpha(t) + \frac{t-t_{n-1}}{\tau}\alpha(t_n) - \frac{t-t_n}{\tau}\alpha(t_{n-1})\notag\\
&= \int_t^{t_n}\frac{(t-t_{n-1})(t_n-s)}{\tau}\alpha''(s)~\mbox{d}s - \int_t^{t_{n-1}}\frac{(t-t_n)(t_{n-1}-s)}{\tau}\alpha''(s)~\mbox{d}s.\label{R4 est}
\end{align}
Applying the Cauchy-Schwarz inequality and the estimate \eqref{I_N convergence} gives that
\begin{align}
\|\mathcal{R}_4(t)\|_{\mathcal{N}}
&\leq C\tau^{\frac{3}{2}}\|\nabla_{\mathcal{N}}\cdot \beta(\nabla_{\mathcal{N}} u(t))\|_{H^{2}(t_{n-1},t_{n+1};L^2_h)}\leq C\tau^{\frac{3}{2}}\|u_e(t)\|_{H^{2}(t_{n-1},t_{n+1};H^2)}.
\end{align}
Then we arrive at
\begin{align*}
&~\sum_{i=1}^4 \|R_i(t)\|_{L^2(t_n,t_{n+1};L_h^2)}^2 \\
=~& \sum_{i=1}^3 \|\widetilde{\mathcal{R}}_i(t)\|_{L^2(t_n,t_{n+1};L^2)}^2 + \|R_4(t)\|_{L^2(t_n,t_{n+1};L_h^2)}^2\\
\leq~& C(h^{2m}+\tau^4)(\|u_e(t)\|_{H^{1}(t_n,t_{n+1};H^{m+4})} + \|u_e(t)\|_{H^{2}(t_{n-1},t_{n+1};H^2)}^2).
\end{align*}
Substituting the above estimate into \eqref{gronwall-mid} yields
\begin{align}
&\omega (t_{n+1}) - \omega (t_n) - 4\tau\omega (t_{n+1})  \nonumber \\
\leq ~& \frac{C(1+h)(2+\varepsilon^2)\tau}{\varepsilon^2}\left(\|e(t_n)\|_{\mathcal{N}}^2 + \|e(t_{n-1})\|_{\mathcal{N}}^2\right)\notag\\
&+C(h^{2m}+\tau^4)(\|u_e(t)\|_{H^{1}(t_n,t_{n+1};H^{m+4})} + \|u_e(t)\|_{H^{2}(t_{n-1},t_{n+1};H^2)}^2). \label{gronwall-mid2}
\end{align}
A summation of the above inequality from $1$ to $n$ results in
\begin{align}
\omega (t_{n+1}) - \omega (t_1) - 4\tau\sum_{i=2}^{n+1}\omega (t_i)
&\leq \frac{C(1+h)(2+\varepsilon^2)}{2\varepsilon^2}\tau \sum_{i=0}^{n}\|e(t_i)\|_{\mathcal{N}}^2 + C(h^{2m}+\tau^4), \nonumber
\end{align}
where $C$ is dependent of $\|u_e\|_{W^{1,4}(0,T;H^{m+4})}, ~\|u_e(t)\|_{H^{2}(0,T;H^2)}$. Since $4\tau \leq \frac{1}{2}$, we have
\begin{align}
\frac{1}{2}\omega (t_{n+1})
&\leq  \left[ \frac{C(1+h)(2+\varepsilon^2)}{\varepsilon^2} + 4\right]\tau\sum_{i=0}^{n}\omega (t_i) + C(h^{2m} + \tau^4) + \omega (t_1). \label{sETDMs2: gronwall}
\end{align}
As for the initial step, subtracting \eqref{sETDMs2: discrete-init} from \eqref{MBE spatial discrete}, we see that: for $0\leq t\leq \tau$,
\begin{align}
\frac{{\rm d} e(t)}{{\rm d} t} + A\tau^2\frac{{\rm d}\Delta^2_{\mathcal{N}} e(t)}{{\rm d} t} + \varepsilon^2\Delta_{\mathcal{N}}^2 e(t)
&= - \nabla_{\mathcal{N}} \cdot \left[\beta(\nabla_{\mathcal{N}}u_0) - \beta(\nabla_{\mathcal{N}}u_s^0)\right] + \mathcal{R}^1(t) \nonumber \\
&= \mathcal{R}^1(t) ,\label{sETDMs2: error eq - init}
\end{align}
where we have applied the set-up $u_s^0 = u_0$, and $\mathcal{R}^1(t)$ iss the truncation error:
\begin{align}
\mathcal{R}^1(t) &= -\varepsilon^2 (\Delta^2 - \Delta^2_{\mathcal{N}})u(t) + A\tau^2\frac{{\rm d}\Delta^2_{\mathcal{N}} u(t)}{{\rm d} t}\nonumber\\
&\quad - \left[\nabla\cdot\beta(\nabla u(t)) - \nabla_{\mathcal{N}}\cdot\beta(\nabla u(t))\right] - \nabla_{\mathcal{N}} \cdot \left[\beta(\nabla u(t)) - \beta(\nabla_{\mathcal{N}} u_0)\right].\label{sETDMs2: truncation - init}
\end{align}
Similarly, we take an inner product with $e(t)$ on both sides,
\begin{align}
&\frac{1}{2}\frac{{\rm d} \|e(t)\|_{\mathcal{N}}^2}{{\rm d} t} + \frac{A\tau^2}{2}\frac{{\rm d}\|\Delta_{\mathcal{N}} e(t)\|_{\mathcal{N}}^2}{{\rm d} t} + \varepsilon^2\|\Delta_{\mathcal{N}} e(t)\|_{\mathcal{N}}^2
\leq \frac{1}{2}\|e(t)\|_{\mathcal{N}}^2 + C(h^{2m} + \tau^4).\label{sETDMs2: error eq 2 - init}
\end{align}
Repeating the process from \eqref{sETDMs2: error eq 3} to \eqref{sETDMs2: gronwall} for $\omega (t)$ in $[0, \tau]$, since $\omega (0) = 0$, we get
\begin{align*}
e^{-\tau}\omega (t_1)\leq C\tau (h^{2m} + \tau^4),
\end{align*}
which again leads to
\begin{align}
\omega (t_1)\leq Ce^{\tau}(\tau h^{2m} + \tau^5). \label{sETDMs2: init error}
\end{align}
A substitution of \eqref{sETDMs2: init error} into \eqref{sETDMs2: gronwall} implies that
\begin{align}
\frac{1}{2}\omega (t_{n+1})
&\leq  \left[ \frac{C(1+h)(2+\varepsilon^2)}{\varepsilon^2} + 4\right]\tau\sum_{i=0}^{n}\omega (t_i) + C(h^{2m} + \tau^4). \label{sETDMs2: gronwall-2}
\end{align}
Applying the discrete Gronwall's inequality and recall that $h$ is bounded above, we get $\omega (t_{n+1}) \leq C_{T, \varepsilon}(h^{2m} + \tau^4)$, i.e.,
\begin{align}
\|e(t_{n+1})\|_{\mathcal{N}}^2 + A\tau^2\|\Delta_{\mathcal{N}} e(t_{n+1})\|_{\mathcal{N}}^2\leq C_{T, \varepsilon,u_0}(h^{2m} + \tau^4),\label{sETDMs2: error result}
\end{align}
where $C_{T, \varepsilon, u_0}$ is a constant dependent on $\|\nabla_{\mathcal{N}}\Delta_{\mathcal{N}} u_0\|_{\mathcal{N}}$, $T$ and $\varepsilon$.
\end{proof}

\begin{rem} Note that in Theorem~\ref{thm:convergence}, the parameter $\varepsilon$ is not involved in the requirement of time step $\tau$, in comparison with  a previous work \cite{li2018second}. This improvement comes from the treatment of the nonlinear term in Lemma~\ref{lm: sETDMs2 nonlinear - middle step}, where we have bounded the nonlinear term by the sum (instead of the product) of $\|\Delta_{\mathcal{N}} e(t)\|$ and $\|e(t)\|$, and thus avoided an $\varepsilon^{-2}$ coefficient of $\|e(t)\|$ after applying Cauchy-Schwarz inequality to obtain an $\varepsilon^2$ coefficient of $\|\Delta_{\mathcal{N}} e(t)\|$. In fact, the restriction can be further relaxed, since we only need $4\tau<1$ in \eqref{sETDMs2: gronwall} for the coefficient of $\omega (t_{n+1})$ to be positive.
\end{rem}

\begin{rem}
There have been some recent works for the linearized energy stable numerical schemes for various 3-D gradient flows, such as a recent paper~\cite{LiD2017} for the Cahn-Hilliard model. In fact, all the results presented in this article could be extended to the 3-D equation, with careful treatments in the Sobolev analysis, although the NSS equation~\eqref{MBE} is of physical interests only in the 2-D case. We may consider an associated extension for other related gradient models in the future works.
\end{rem}

\section{Numerical results}\label{sec: numerical results}
In this section, we first demonstrate various numerical experiments to verify the temporal convergence and energy stability of the sETDMs2 scheme \eqref{sETDMs2}. In particular, a comparison with the the ETDMs2 scheme (in \cite{Ju17}) are performed, with $\kappa = \frac{1}{8}$ in ETDMs2, and $A = \frac{1}{8}, 10^{-2}, \frac{2+\sqrt{3}}{6}$ in sETDMs2. And also, some long time simulation results computed by the sETDMs2 are presented, and some long time characteristics of the numerical solution are studied, such as the energy decay rate, the average surface roughness and average slope.

\subsection{Temporal Convergence}
Throughout this subsection, we consider solving \eqref{MBE} with $\Omega = [0,2\pi]^2$, $\varepsilon^2 = 0.01$, $T = 1$, and the initial value $u_e(0) = \sin x \cos y $ on the uniform $N\times N$ mesh with $N = 256$. The time step size is set to be $\tau =  0.005 * [2^{-1}, 2^{-2}, 2^{-3}, 2^{-4}, 2^{-5}, 2^{-6}]$. With an additional time-dependent forcing term $g(t)$, the exact solution to equation \eqref{MBE} is given by $ u_e(t) = \sin x \cos y \cos t$:
\begin{align*}
g(t) 	&= -\sin x\cos y\sin t + 4\varepsilon^2\sin x\cos y\cos t + \frac{-2\sin x\cos y\cos t}{1+\frac{(\cos t)^2}{2}(1+\cos 2x\cos 2y)}\\
		&\quad + \frac{(\cos t)^3(\cos x\cos y\sin 2x\cos 2y - \sin x\sin y\cos 2x\sin 2y)}{[1+\frac{(\cos t)^2}{2}(1+\cos 2x\cos 2y)]^2}.
\end{align*}

To solve the molecular epitaxial growth equation with a forcing term, we applied both the sETDMs2 and the ETDMs2 schemes. For each algorithm, we computed the relative error ${\left\|u^{N_t} - u_{\tau}^{N_t}\right\|_{\mathcal{N}}}/{\|u^{N_t}\|_{\mathcal{N}}}$ between the exact solution and numerical solutions. Results are displayed in Figure~\ref{fig: 1}, from which the second order temporal convergence is clear observed. Also, different choices of the stabilization coefficient $A$ in the sETDMs2 scheme don't have an evident impact on the reported error values.

\begin{figure}[ht]
\centering
\includegraphics[width = 0.8\textwidth]{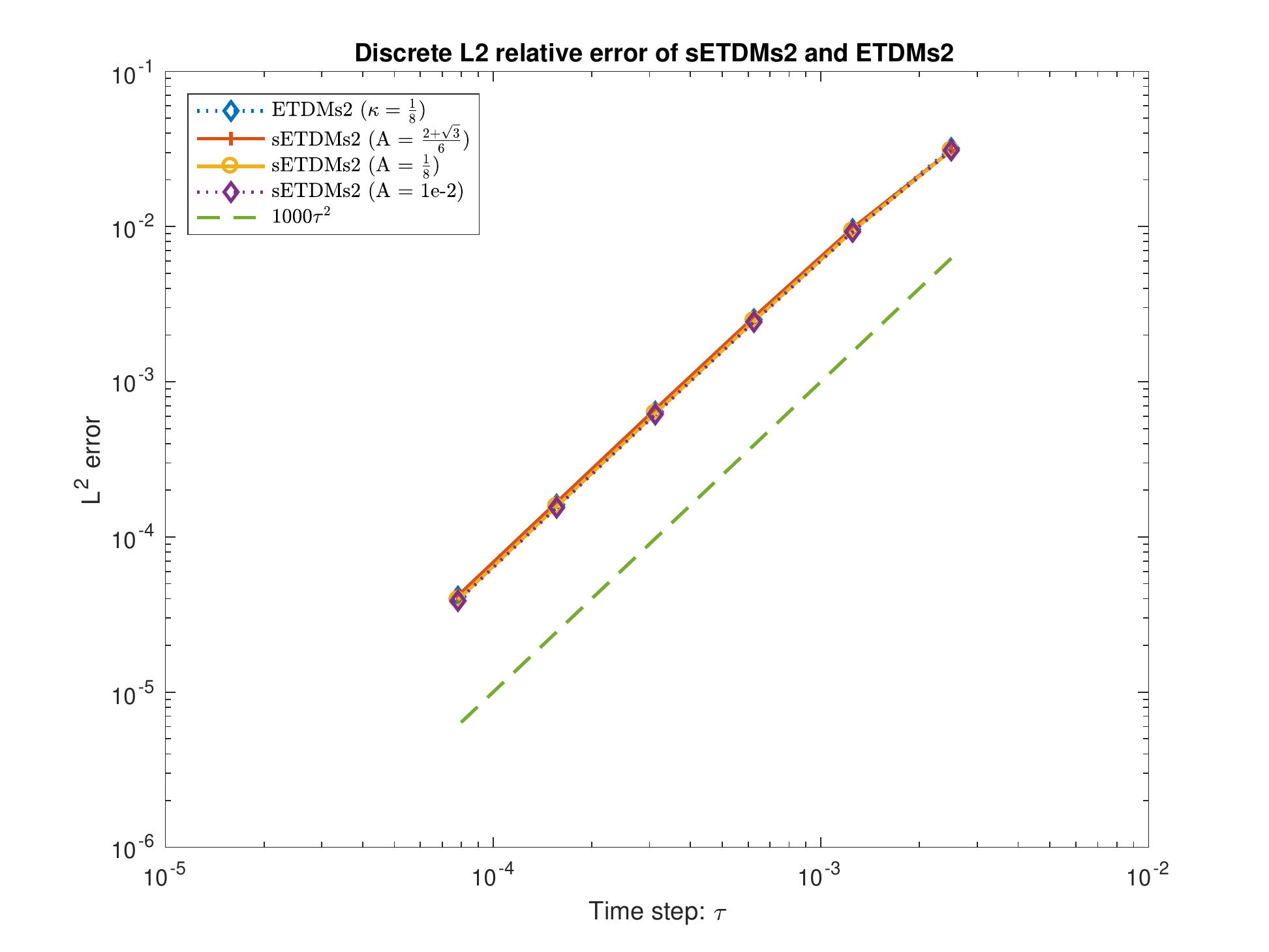}
\caption{Temporal convergence rates of the sETDMs2, ETDMs2 schemes.}
\label{fig: 1}
\end{figure}

\subsection{Energy dissipation}
In this subsection, we set $\Omega = [0,12.8]^2$, $\varepsilon^2 = 0.005$, $T = 25000$ and use a random initial data. The reason for such a choice comes from a subtle fact that, a random initial data leads to a solution with various wave frequencies, so that more detailed and interesting phenomenon associated with different wave frequencies are presented in the coarsening process, while a smooth initial value may yield a solution with trivial structure in a short time.

We use a coarser uniform mesh with $N = 128$, and set time step size $\tau = 0.001$ for $t<200$, $\tau = 0.01$ for $200\leq t < 1000$, $\tau = 0.02$ for $1000\leq t < 2000$, and $\tau = 0.04$ for $t\geq 2000$. Whenever the time step is changed, previous step's solution is used as the initial data and we reuse the initial sETDMs2 step for a start. Figure~\ref{fig: 2} shows the snapshots of the numerical solution $u$ at time $t = $ 1, 1500, 5000, 15000, respectively. It can be observed that the solution has saturated to a one-hill-one-valley structure at the final time. Specially, when $A$ in sETDMs2 is lower than the restriction in Theorem~\ref{thm: energy stab}, the numerical solution still converges to the stable state.

\begin{figure}[ht]
\centering
	\noindent\makebox[\textwidth][c] {
  		\begin{minipage}{0.25\textwidth}
  	 	\includegraphics[width=\textwidth]{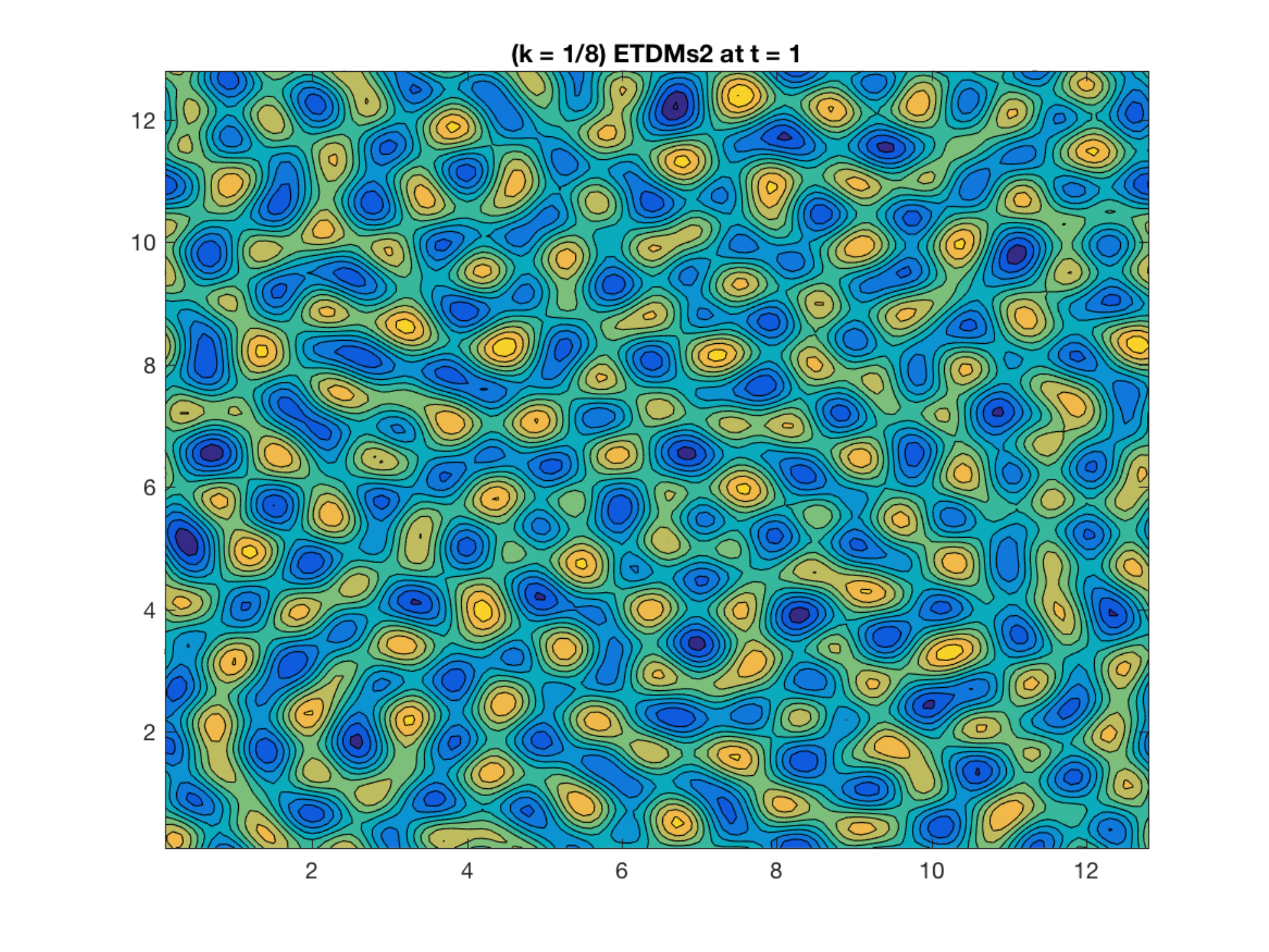}
  		\end{minipage}
 		\begin{minipage}{0.25\textwidth}
  	 	\includegraphics[width=\textwidth]{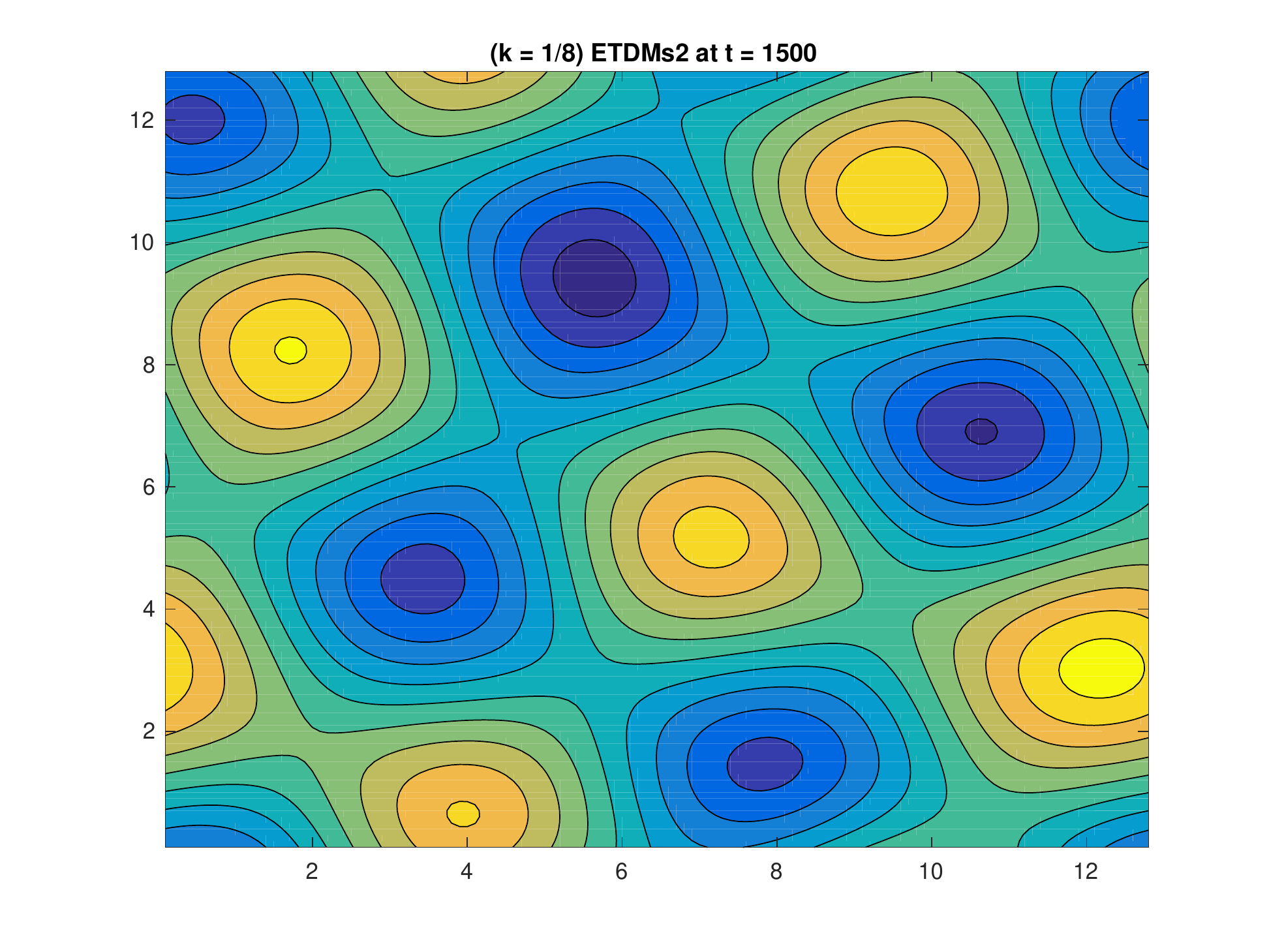}
		 \end{minipage}
  		\begin{minipage}{0.25\textwidth}
  	 	\includegraphics[width=\textwidth]{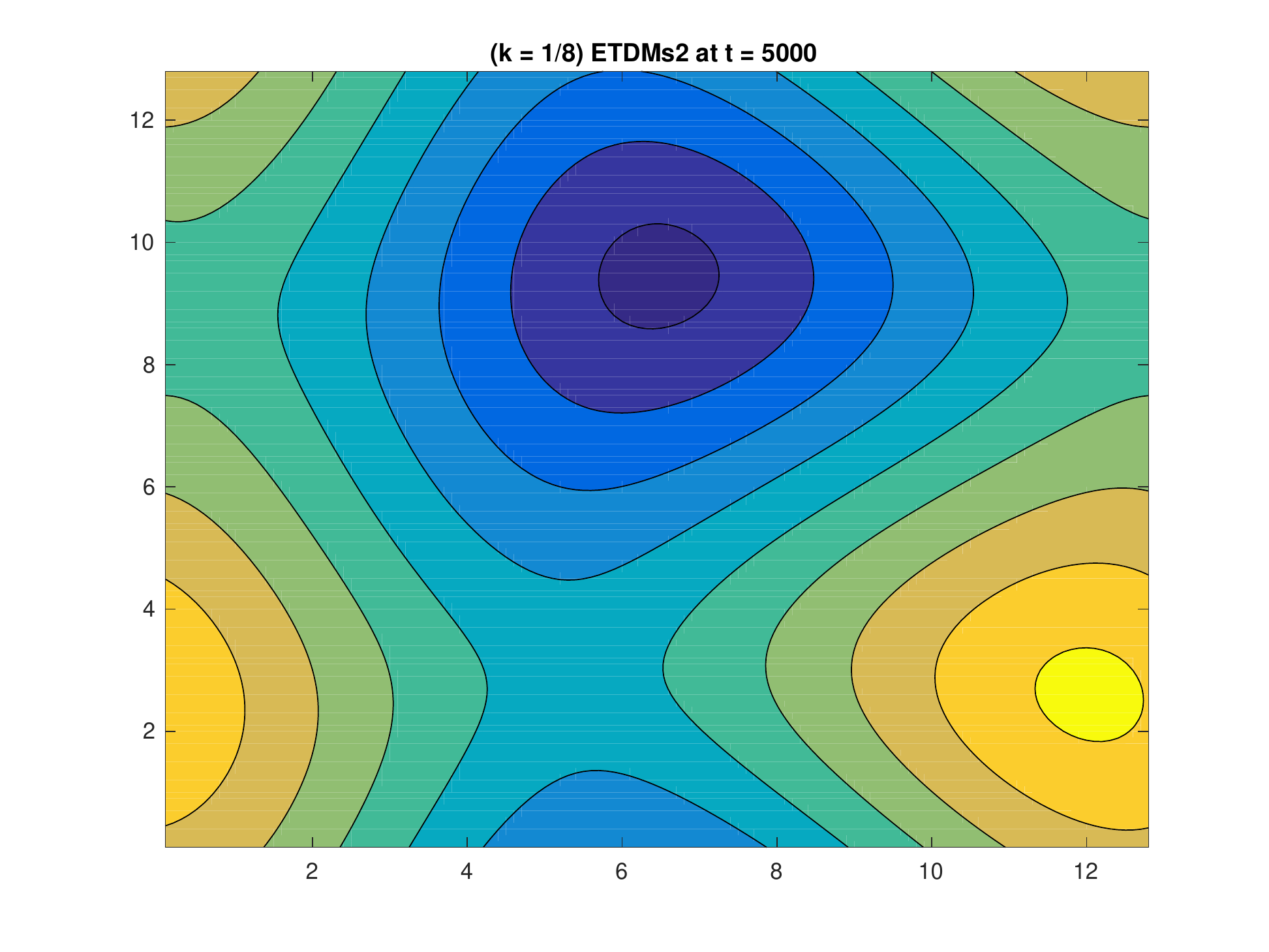}
  		\end{minipage}
 		\begin{minipage}{0.25\textwidth}
  	 	\includegraphics[width=\textwidth]{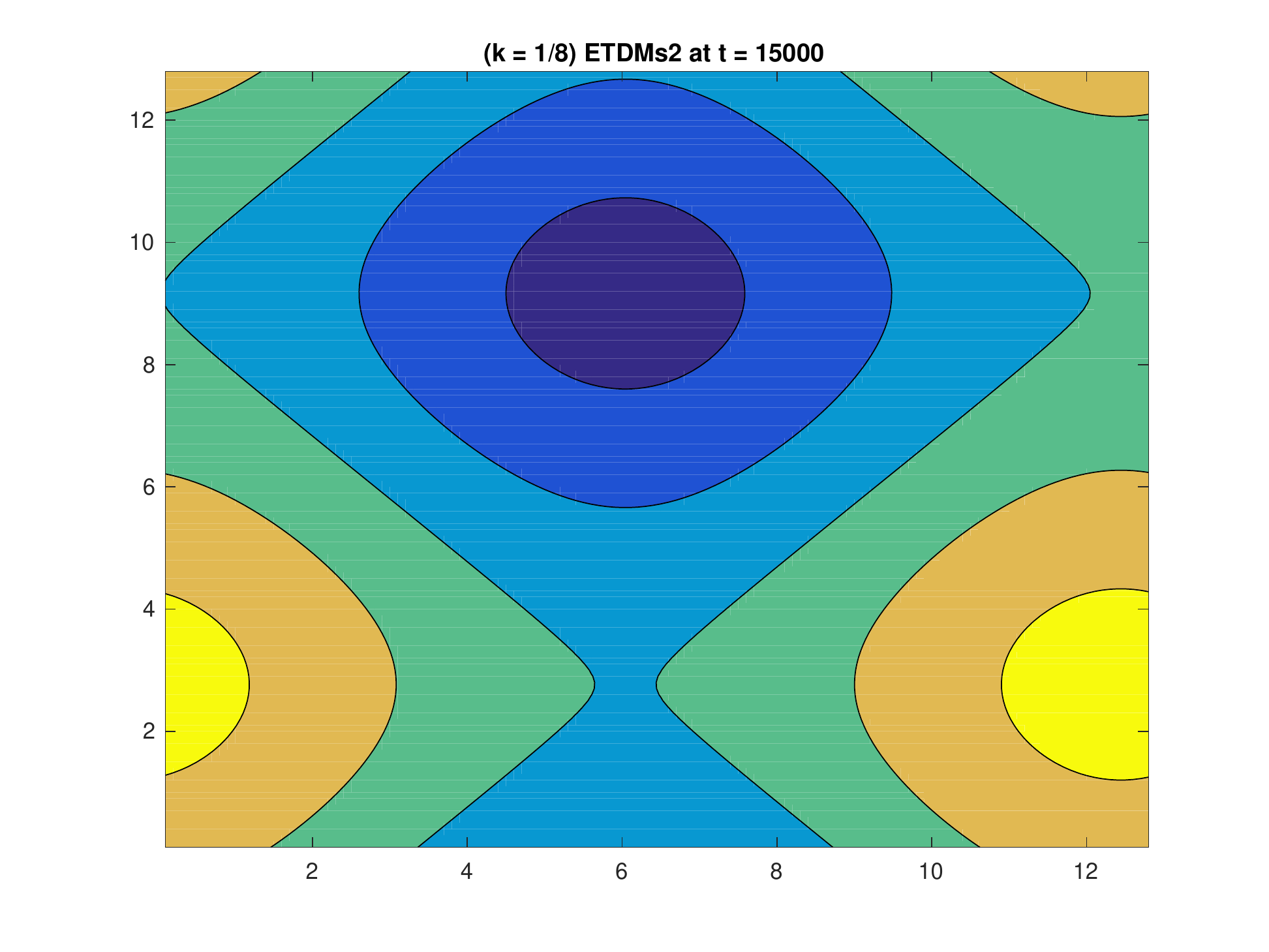}
		 \end{minipage}
	 }
	\noindent\makebox[\textwidth][c] {
  		\begin{minipage}{0.25\textwidth}
  	 	\includegraphics[width=\textwidth]{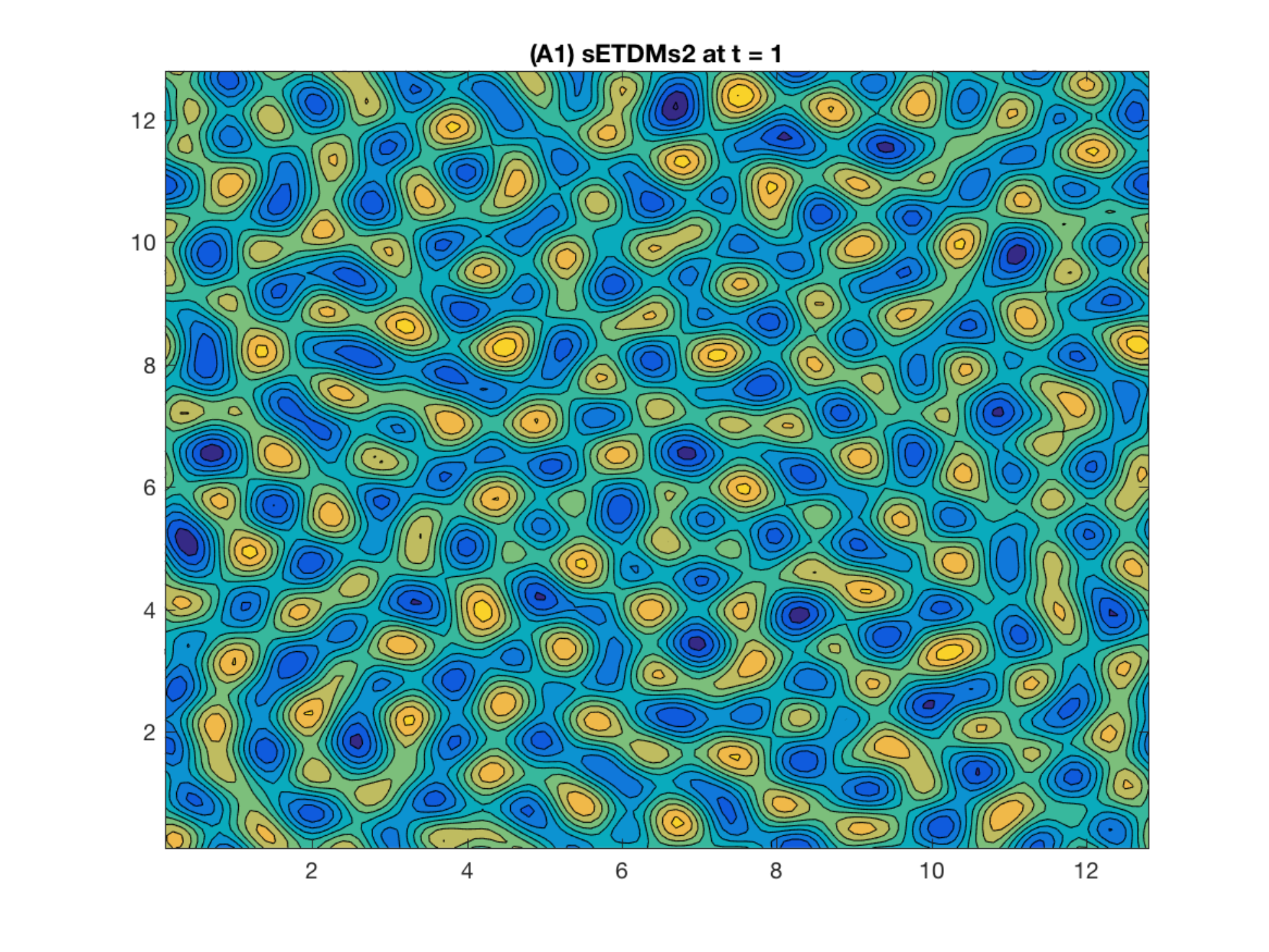}
  		\end{minipage}
 		\begin{minipage}{0.25\textwidth}
  	 	\includegraphics[width=\textwidth]{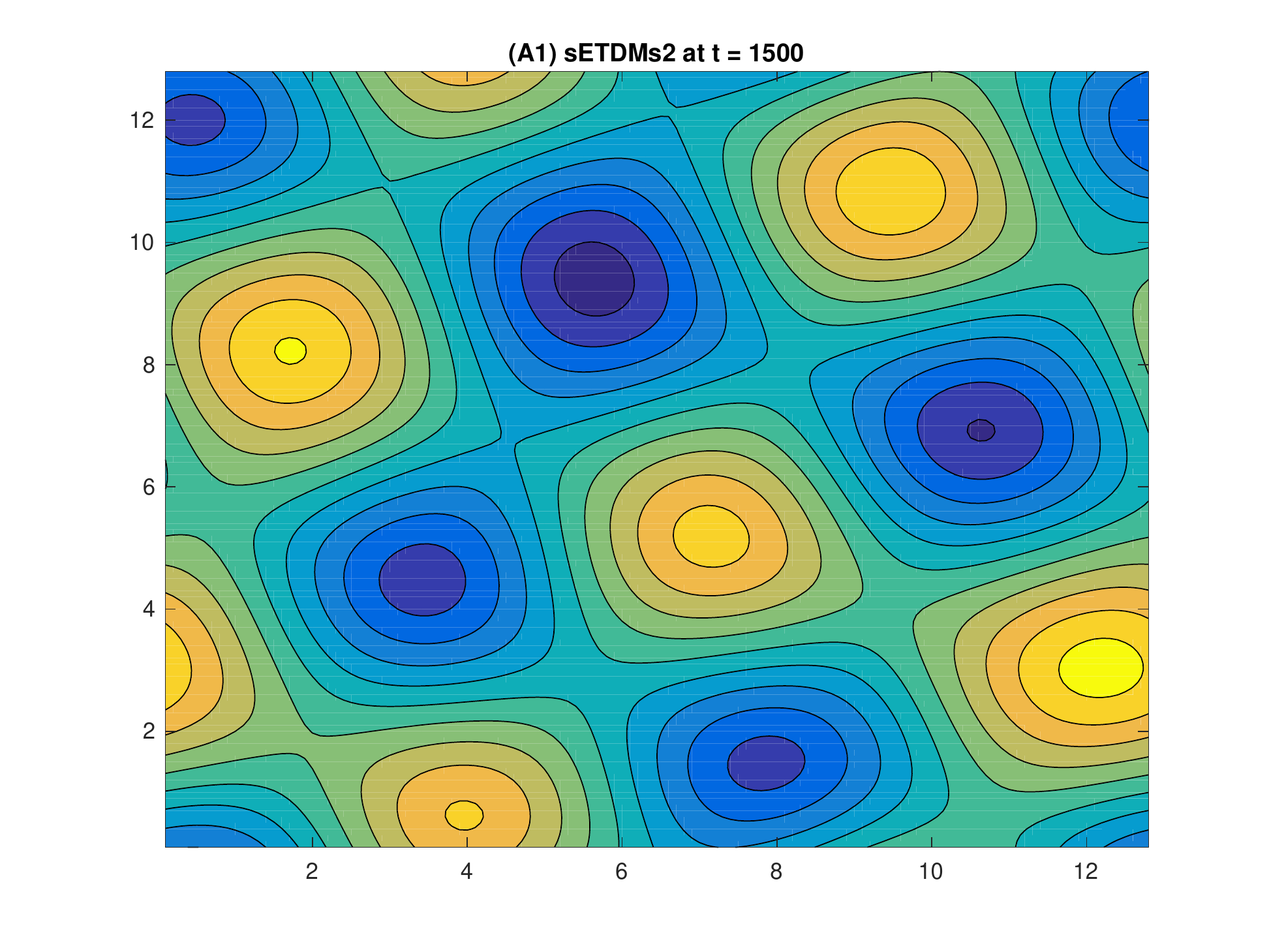}
		 \end{minipage}
  		\begin{minipage}{0.25\textwidth}
  	 	\includegraphics[width=\textwidth]{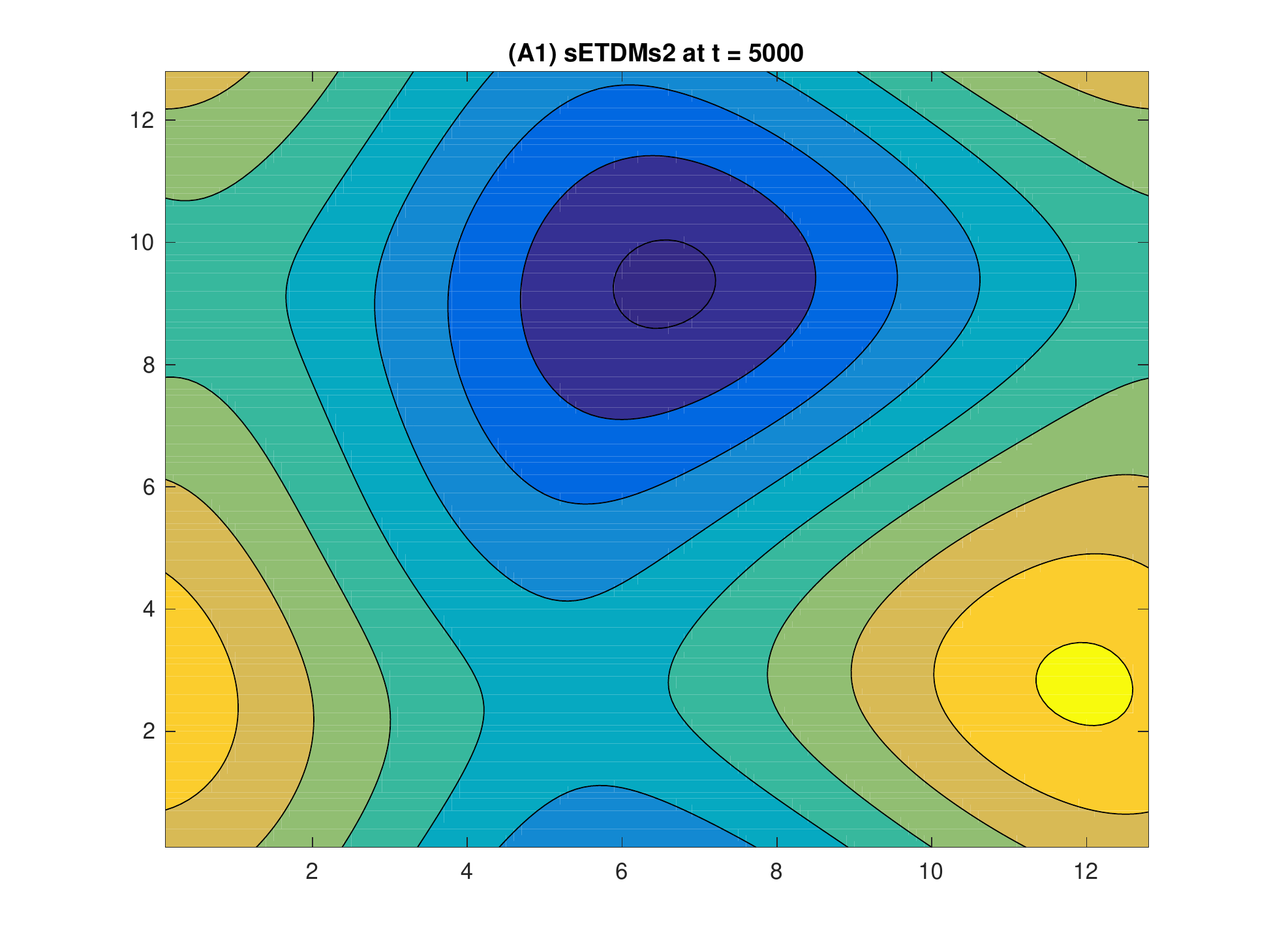}
  		\end{minipage}
 		\begin{minipage}{0.25\textwidth}
  	 	\includegraphics[width=\textwidth]{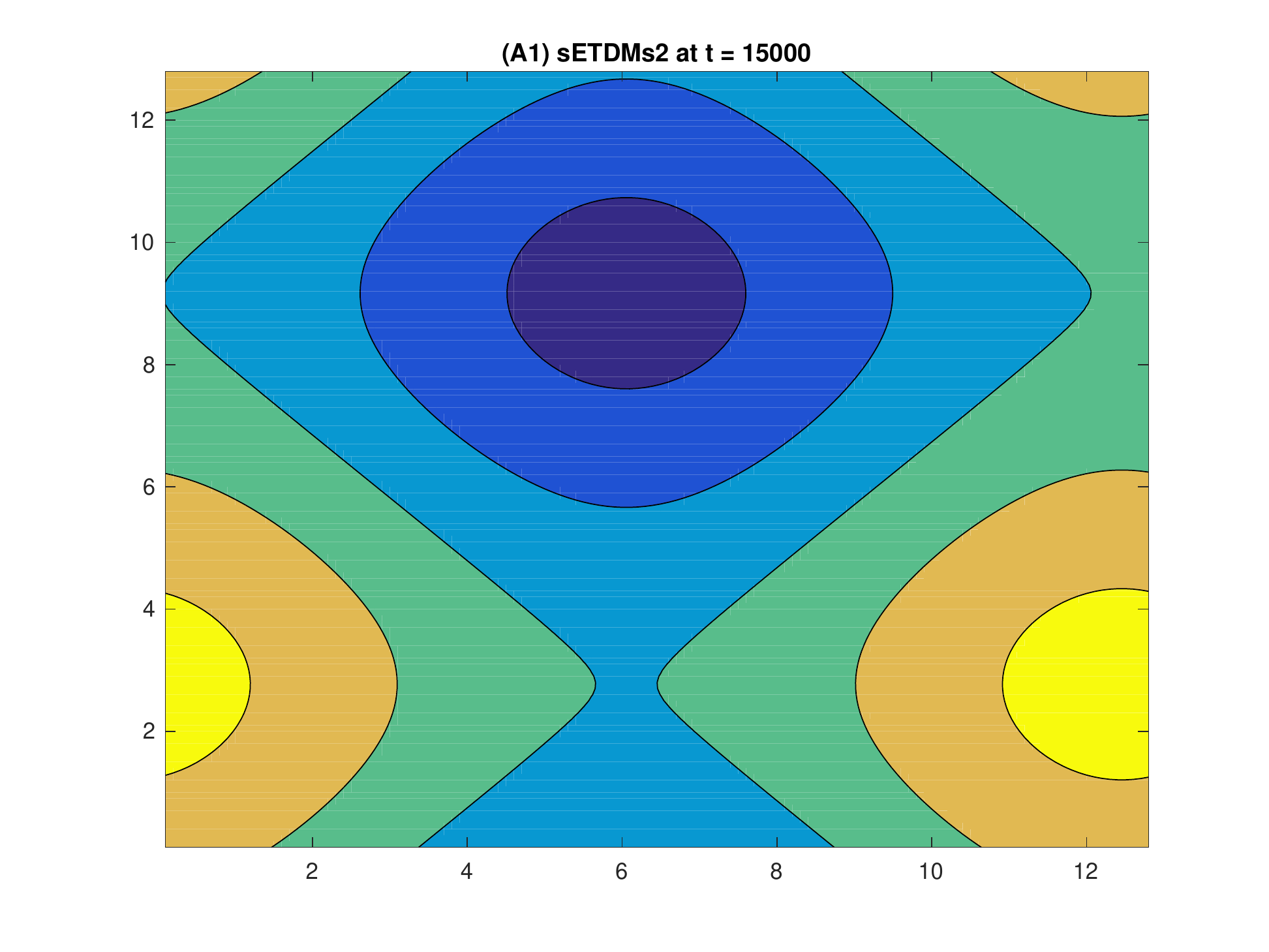}
		 \end{minipage}
	 }
	\noindent\makebox[\textwidth][c] {
  		\begin{minipage}{0.25\textwidth}
  	 	\includegraphics[width=\textwidth]{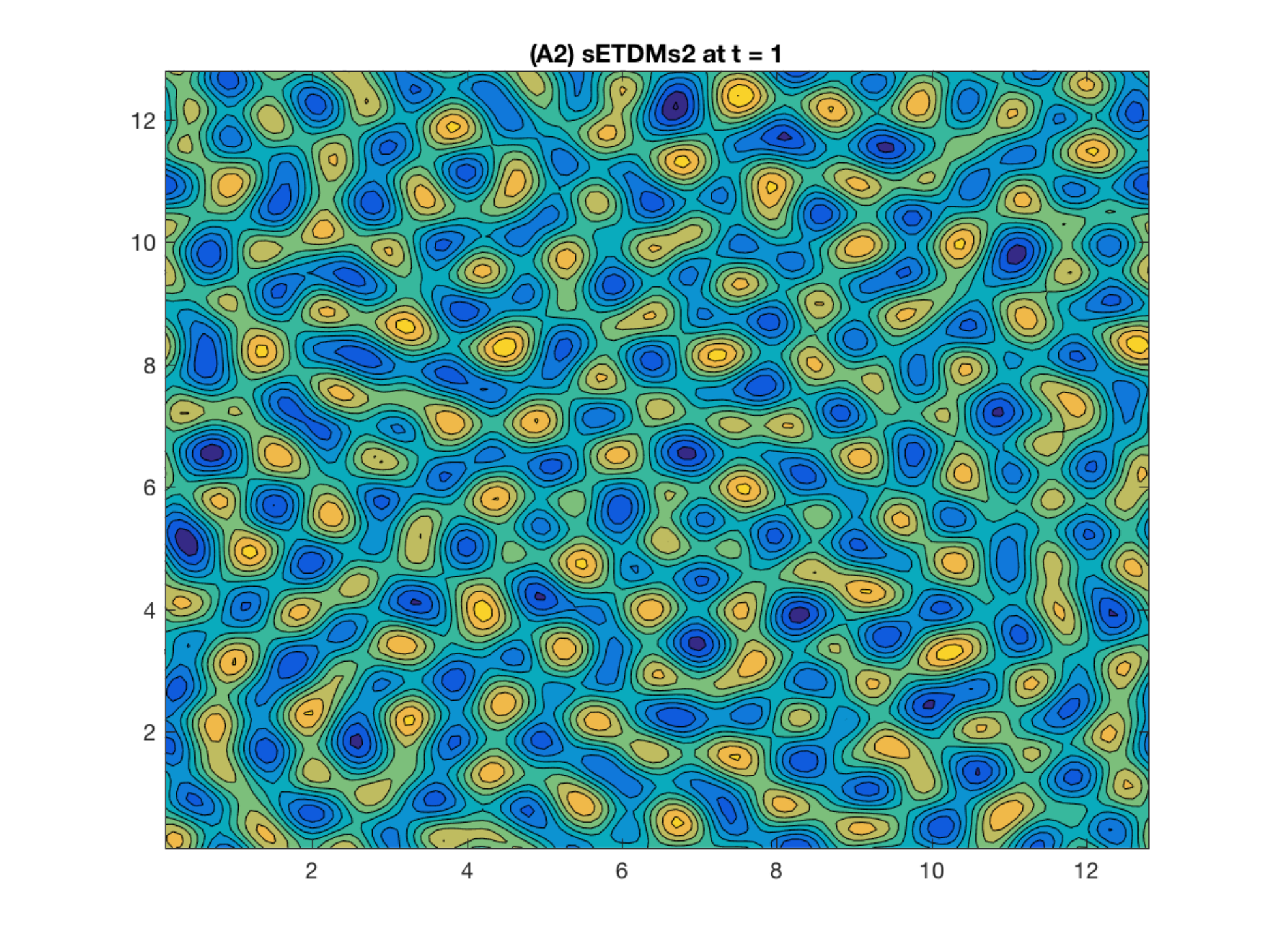}
  		\end{minipage}
 		\begin{minipage}{0.25\textwidth}
  	 	\includegraphics[width=\textwidth]{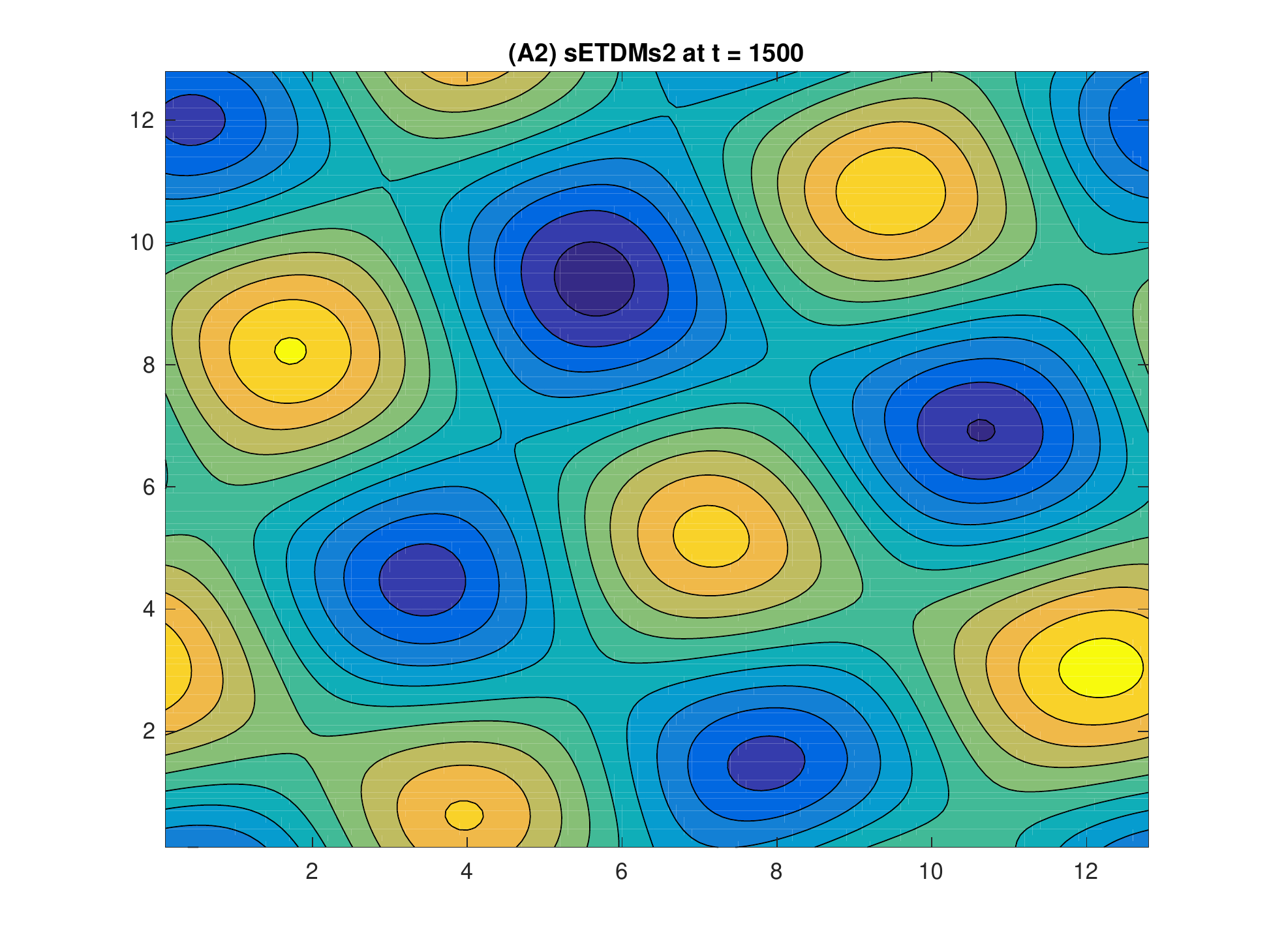}
		 \end{minipage}
  		\begin{minipage}{0.25\textwidth}
  	 	\includegraphics[width=\textwidth]{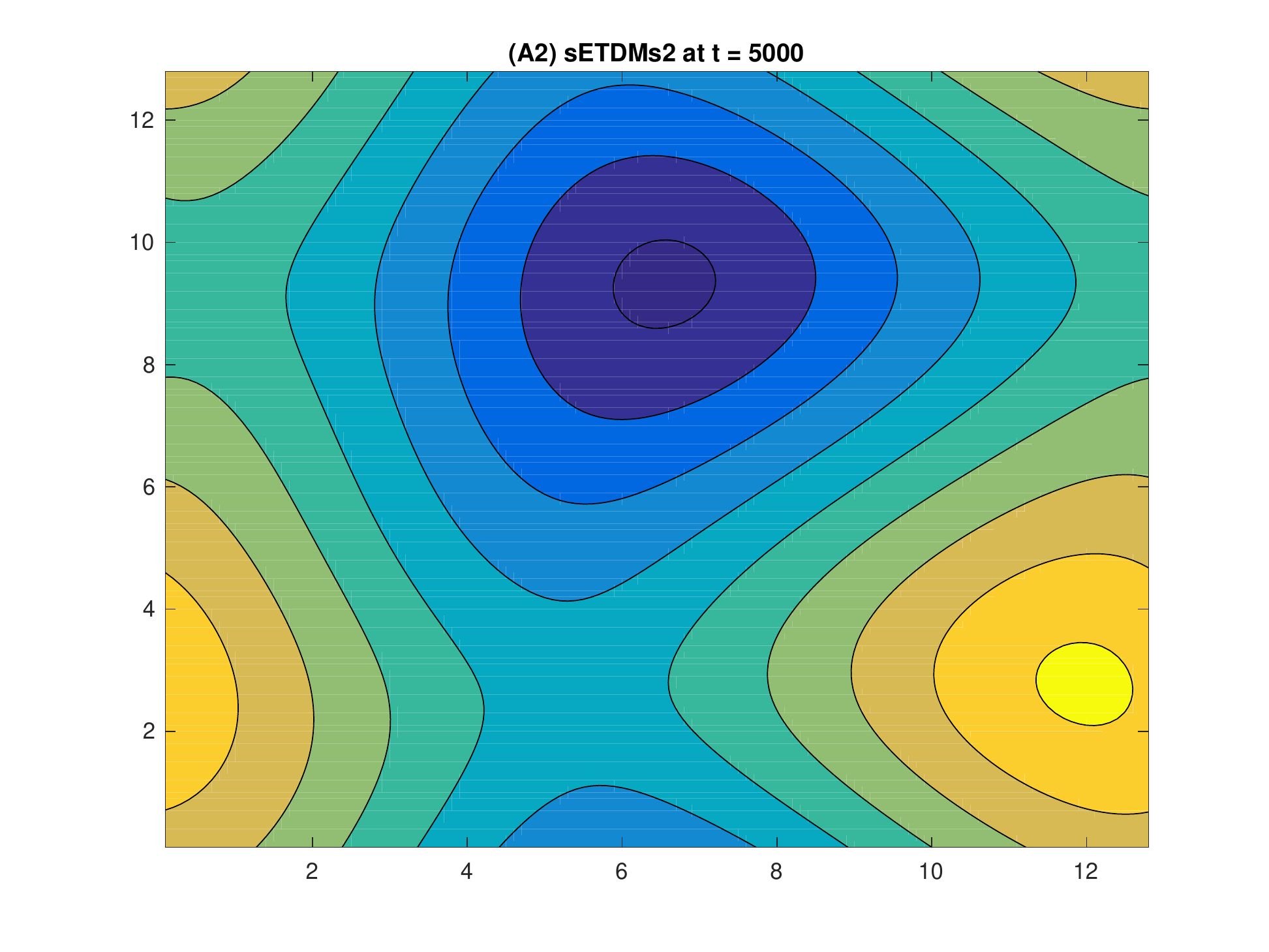}
  		\end{minipage}
 		\begin{minipage}{0.25\textwidth}
  	 	\includegraphics[width=\textwidth]{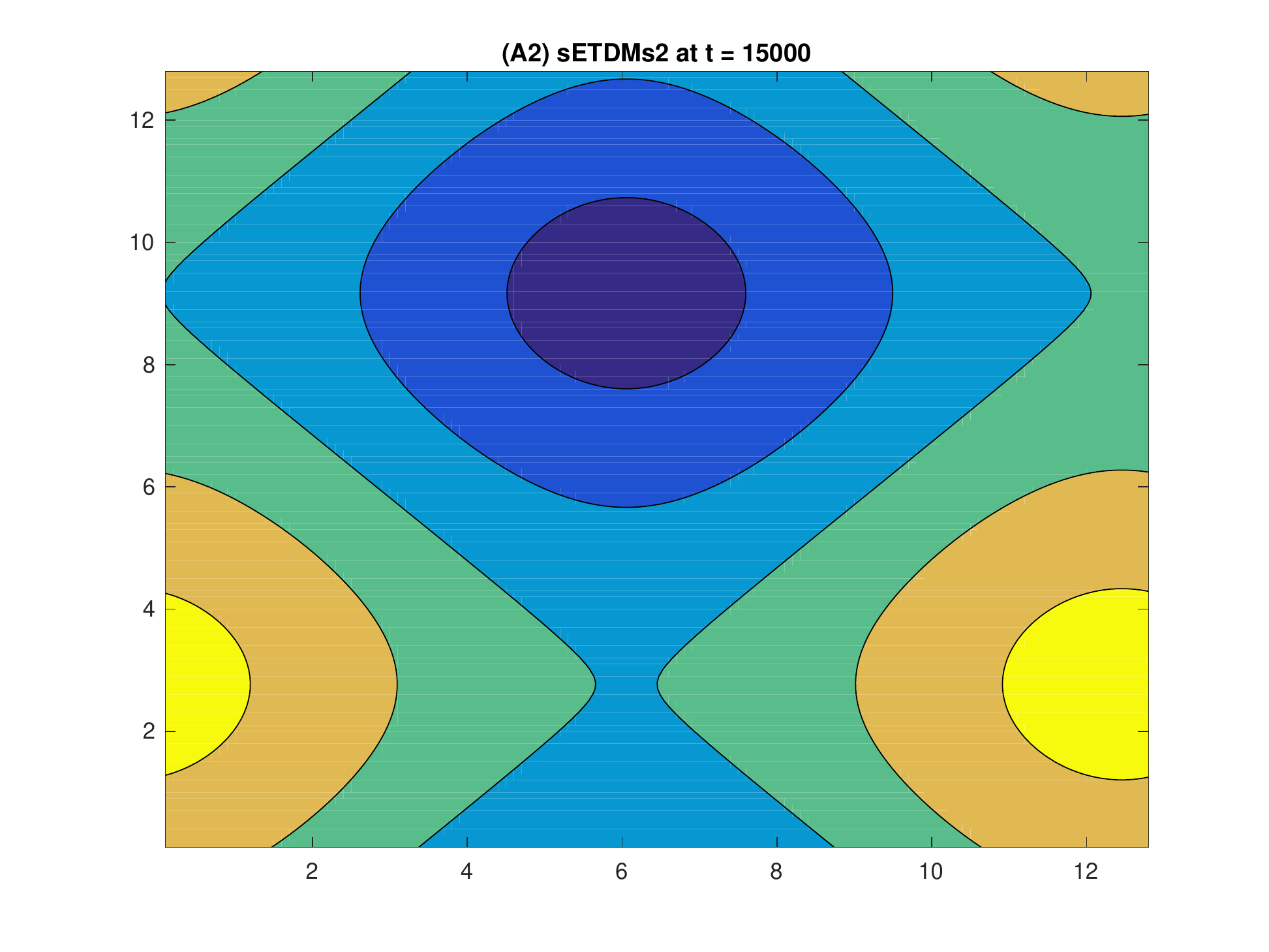}
		 \end{minipage}
	 }
	\noindent\makebox[\textwidth][c] {
  		\begin{minipage}{0.25\textwidth}
  	 	\includegraphics[width=\textwidth]{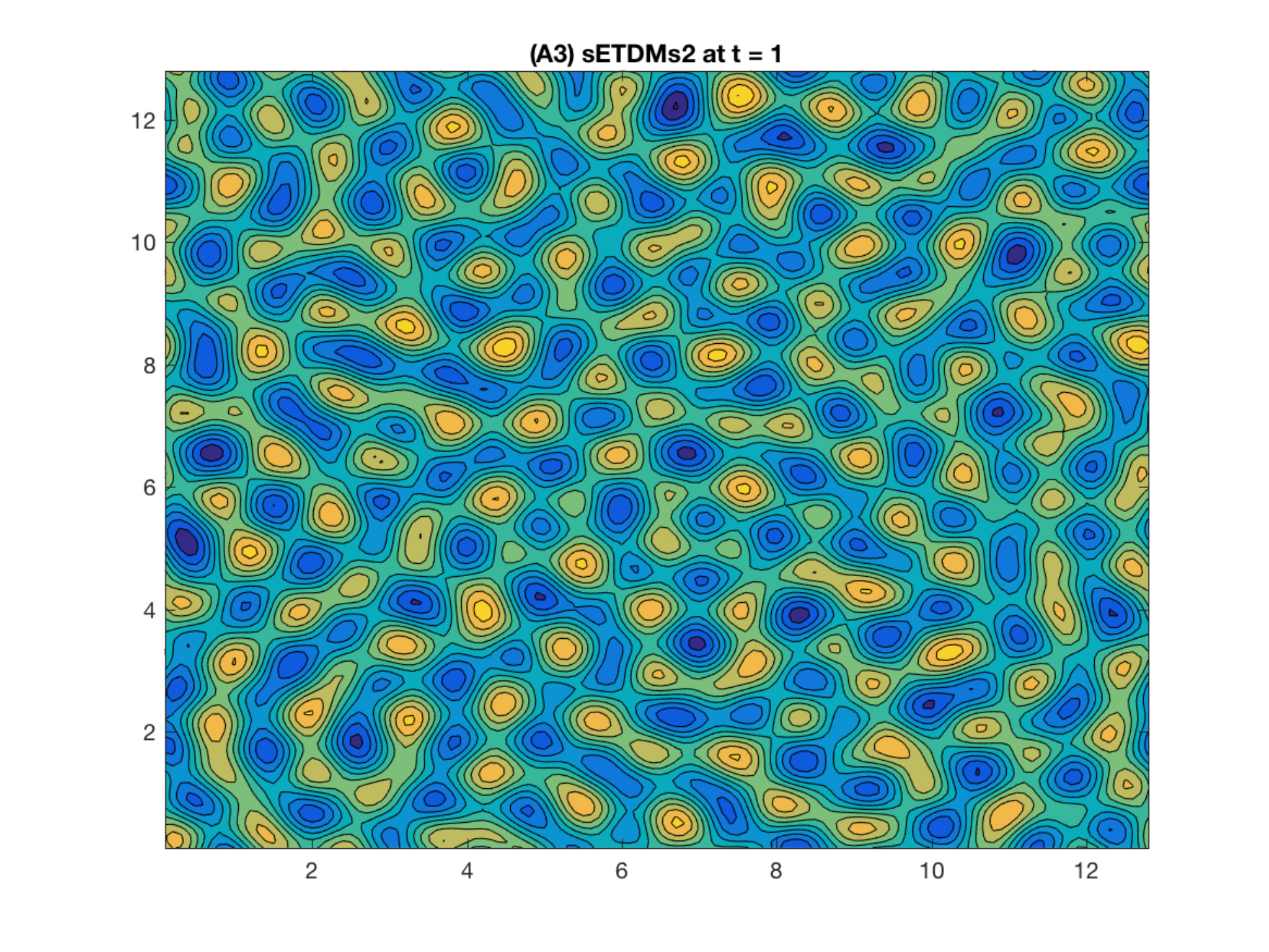}
  		\end{minipage}
 		\begin{minipage}{0.25\textwidth}
  	 	\includegraphics[width=\textwidth]{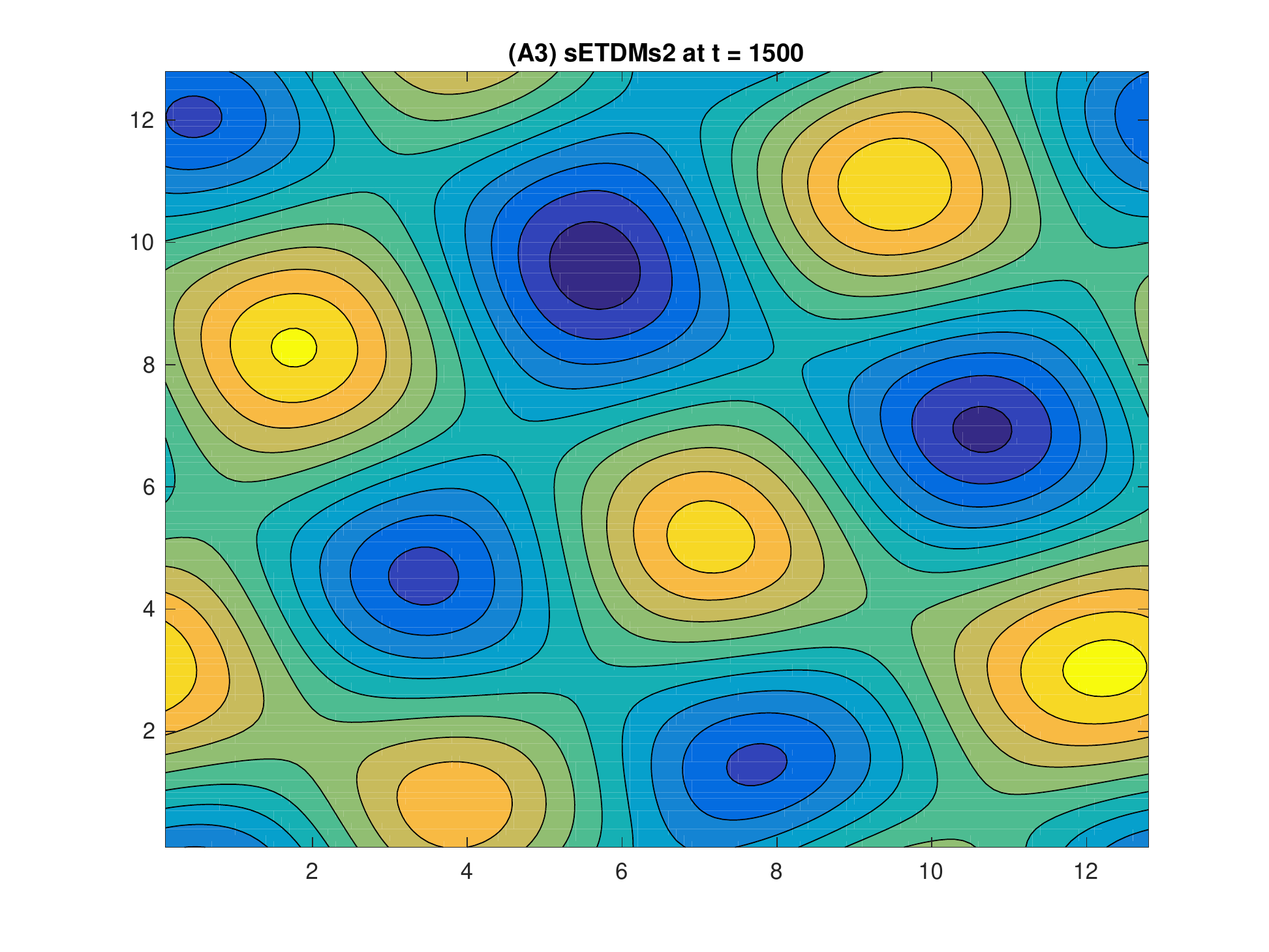}
		 \end{minipage}
  		\begin{minipage}{0.25\textwidth}
  	 	\includegraphics[width=\textwidth]{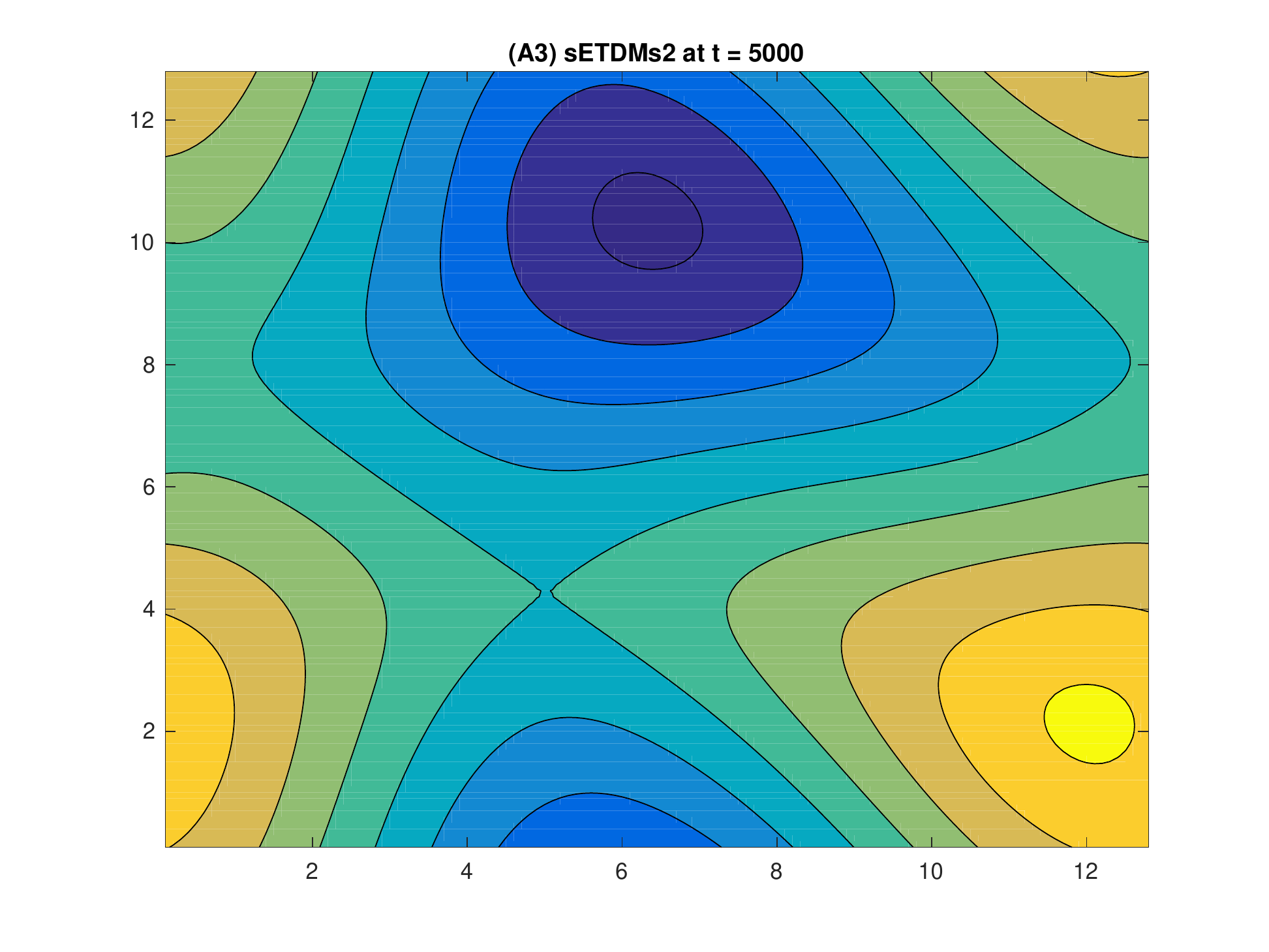}
  		\end{minipage}
 		\begin{minipage}{0.25\textwidth}
  	 	\includegraphics[width=\textwidth]{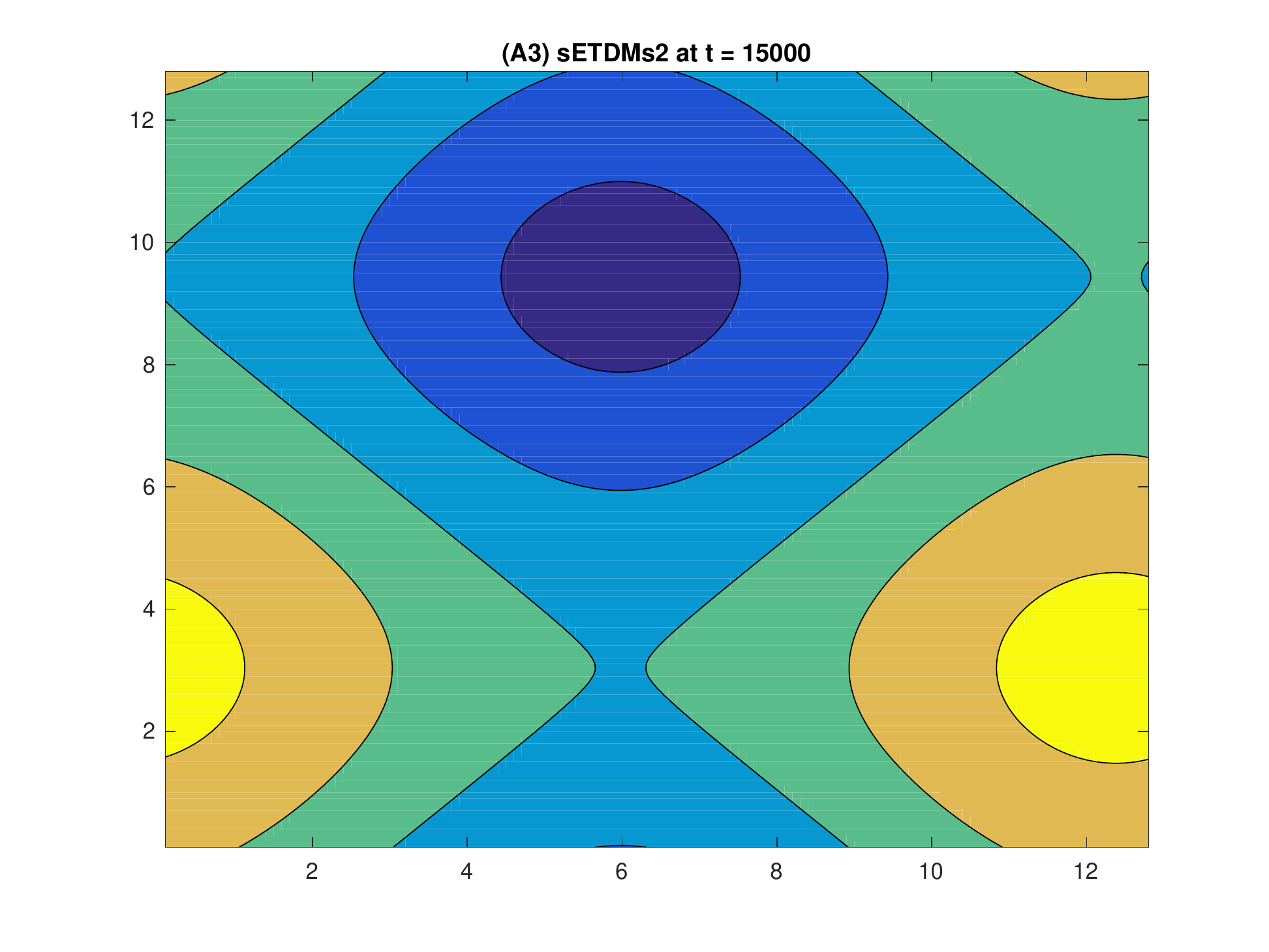}
		 \end{minipage}
	 }
\caption{Snapshots of the numerical solutions. The first row is for ETDMs2 and the next 3 rows are for sETDMs2 with $A = \frac{1}{8}, ~10^{-2}, ~\frac{2+\sqrt{3}}{6}$, respectively.}\label{fig: 2}
\end{figure}

Recall the discrete energy functional in \eqref{MBE energy: discrete}, for convenience we repeat it here:
\begin{equation}
E_{\mathcal{N}}(u)=\left(-\frac{1}{2}\ln(1+|\nabla_{\mathcal{N}}u|^2),1\right)_{\mathcal{N}} + \frac{\varepsilon^2}{2}\|\Delta_{\mathcal{N}} u\|_{\mathcal{N}}^2,~~~\forall u \in {\mathcal{M}}^{\mathcal{N}}.
\end{equation}
Also, consider the average surface roughness $h_{\mathcal{N}}(u)$ and the average slope $m_{\mathcal{N}}(u)$:
\begin{align}
h_{\mathcal{N}}(u,t) &= \sqrt{\frac{h^2}{|\Omega|}\sum_{\mathcal{M}^{\mathcal{N}}} |u(\textbf{x}_{i,j},t)-\bar{u}(t)|^2}, \quad \mbox{with} \quad \bar{u}(t):=\frac{h^2}{|\Omega|}\sum_{\mathcal{M}^{\mathcal{N}}} u(\textbf{x}_{i,j},t).\\
m_{\mathcal{N}}(u,t) &= \sqrt{\frac{h^2}{|\Omega|}\sum_{\mathcal{M}^{\mathcal{N}}} |\nabla u(\textbf{x}_{i,j},t)|^2}.
\end{align}
For the no-slope-selection growth model \eqref{MBE}, it has been proved that  $E_{\mathcal{N}} \sim O(-\ln (t))$, $h_{\mathcal{N}} \sim O(t^{\frac{1}{2}})$ and  $m_{\mathcal{N}} \sim O(t^{\frac{1}{4}})$ as $t\rightarrow\infty$. (See \cite{golubovic1997interfacial, li2003thin, li2004epitaxial} and references therein.) The evolution of $E_{\mathcal{N}}$ is demonstrated in  Figure~\ref{fig: 3}, while $h_{\mathcal{N}}$ and $m_{\mathcal{N}}$ in Figure~\ref{fig: 4}. The linear fitting results are also presented in Figure~\ref{fig: 3} and Figure~\ref{fig: 4}  for the solution of sETDMs2 in time interval $[1, 1000]$, of which the fitting coefficients are displayed in Table~\ref{tab: 1}. As can be seen in Table~\ref{tab: 1}, different choices of $A$ have little impact on the long time characteristics of the numerical solutions.
\begin{table}[ht]
\centering
\caption{Fitting coefficients of long time characteristics for sETDMs2.}\label{tab: 1}
\begin{tabular}{|c|c|c|c|c|c|c|}
\hline
 &  \multicolumn{2}{|c|}{$E_{\mathcal{N}}(t) \approx a \ln (t) + b$}	& \multicolumn{2}{|c|}{$h_{\mathcal{N}}(t) \approx a t^b$}	& \multicolumn{2}{|c|}{$m_{\mathcal{N}}(t) \approx  a t^b$} \\
\cline{2-7}
 														& $a$		& $b$		& $a$		& $b$		& $a$	& $b$	\\
\hline
$A = 1/8$										& -39.36	& -54.36	& 0.4372	& 0.4867	& 2.133	& 0.2614	\\
\hline
$A = 1/100$									& -39.38	& -54.27	& 0.4377	& 0.4865	& 2.132	& 0.2614	\\
\hline
 $A = \frac{2+\sqrt{3}}{6}$ 	& -39.26	& -54.77	& 0.4346	& 0.4879	& 2.137	& 0.2609	\\
\hline
\end{tabular}
\end{table}

\begin{figure}[ht]
\begin{center}
\includegraphics[width =0.8\textwidth]{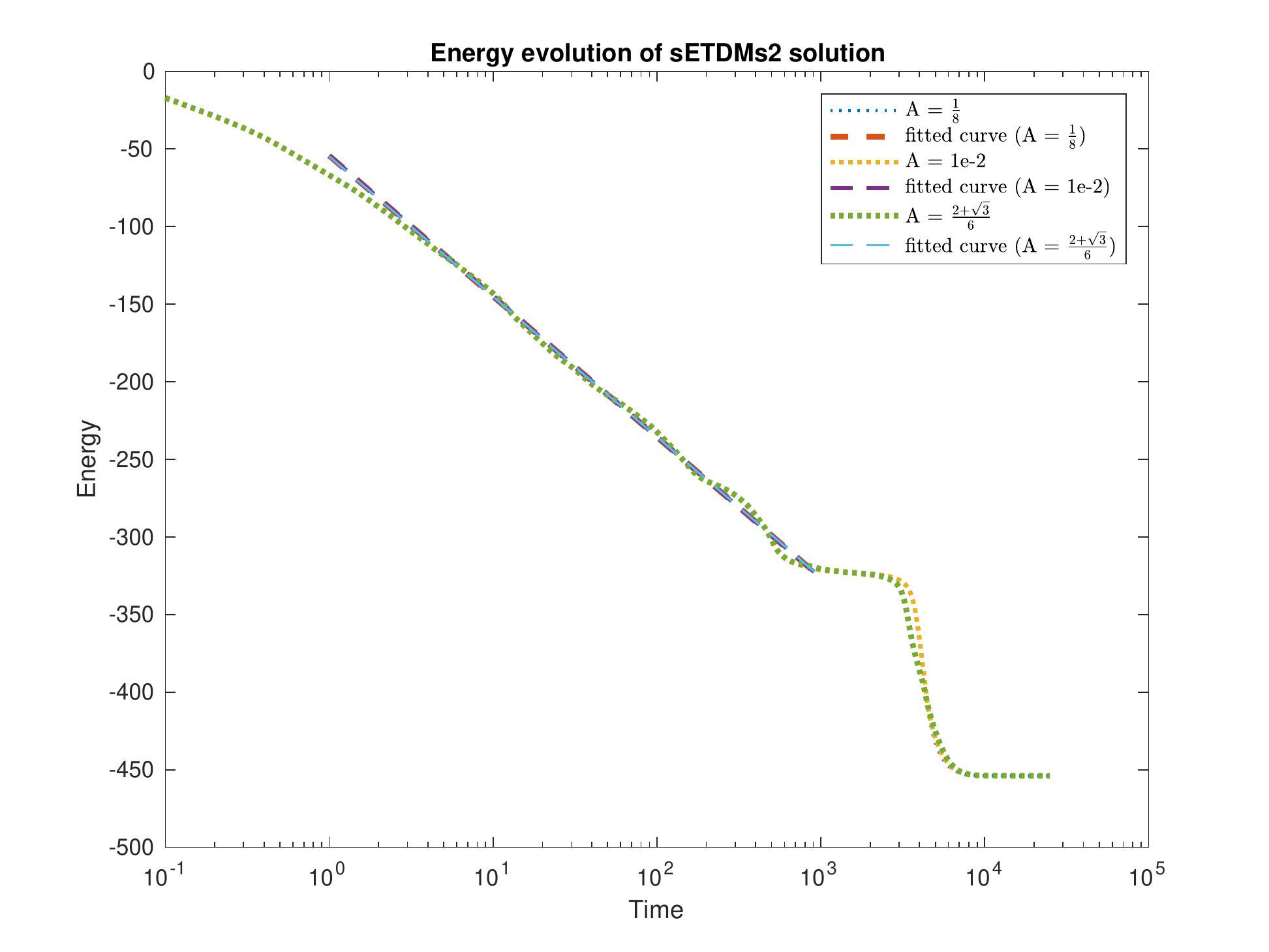}
\caption{Semi-log plot of the energy $E_{\mathcal{N}}$ with $\varepsilon^2 = 0.005$ of sETDMs2 with $A = 	\frac{1}{8}, ~ 10^{-2}, ~\frac{2+\sqrt{3}}{6}$. Fitted line has the form $a \ln (t) + b$, with coefficients shown in Table~\ref{tab: 1}.} \label{fig: 3}
\end{center}
\end{figure}

\begin{figure}[ht]
\begin{center}
\includegraphics[width=0.4\textwidth]{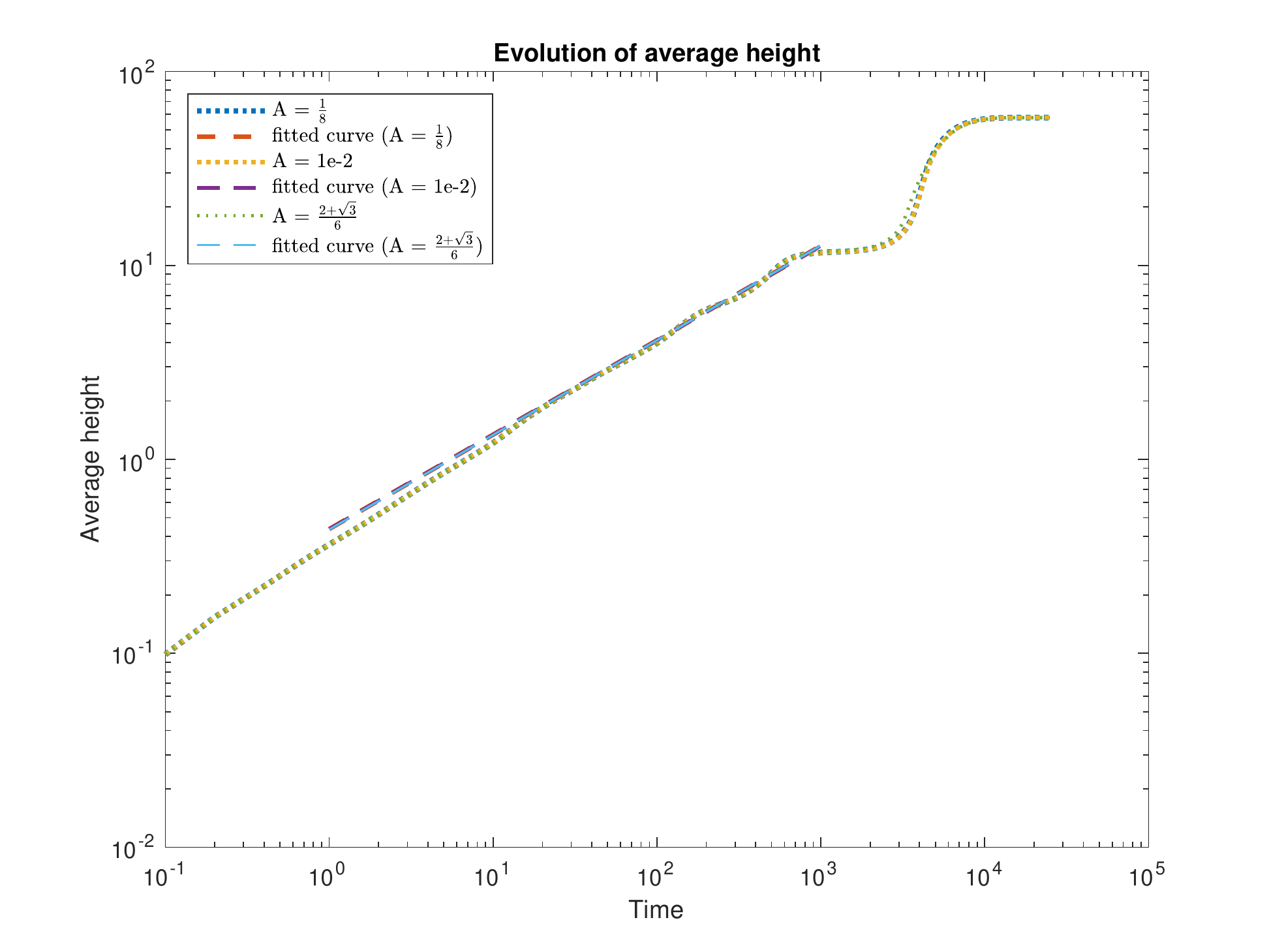}\;
\includegraphics[width=0.4\textwidth]{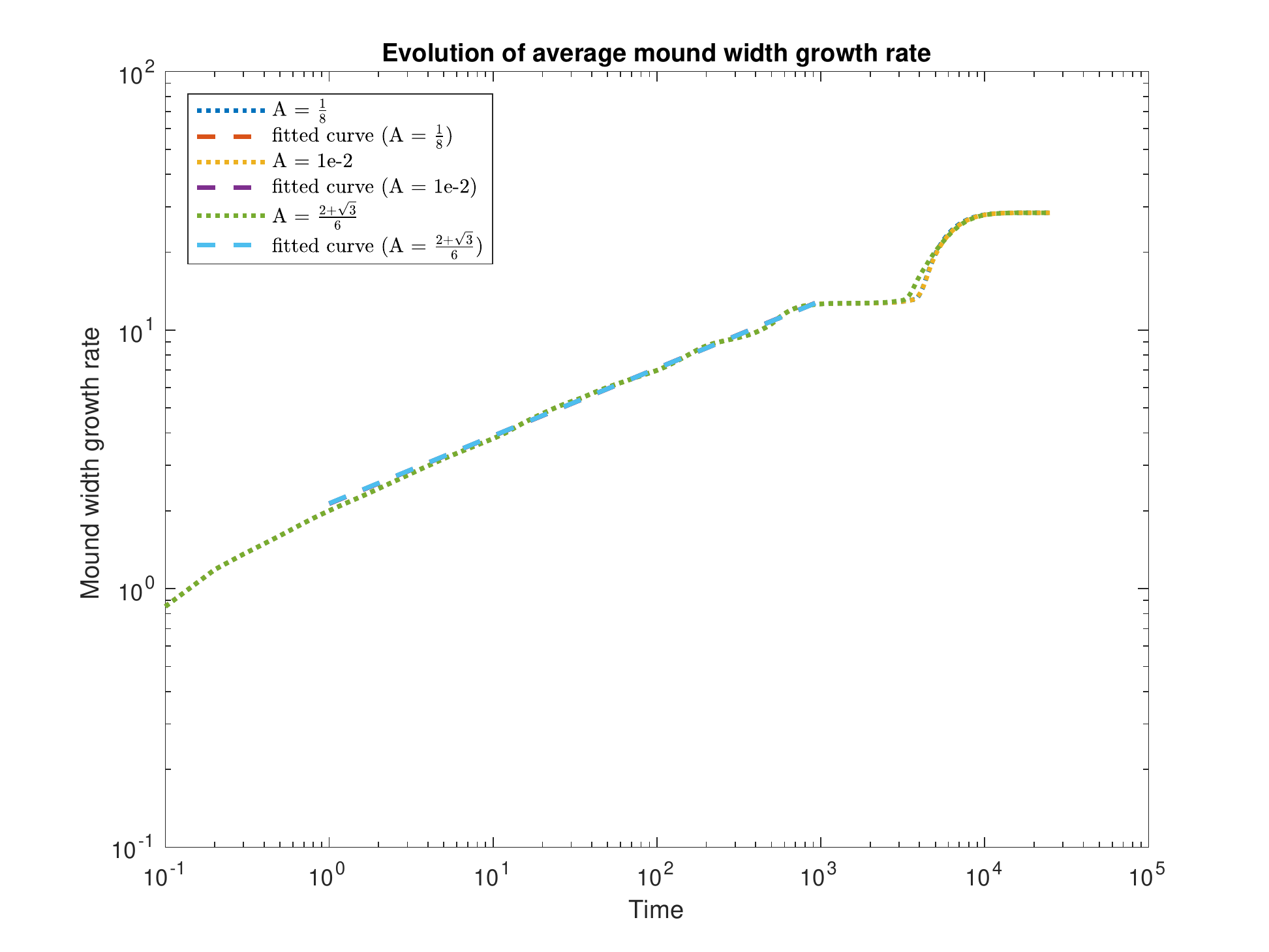}
\caption{The log-log plot of (1) the average surface roughness $h_{\mathcal{N}}$ and (2) the average slope $m_{\mathcal{N}}$ using sETDMs2 with $\varepsilon^2 = 0.005$, $A = \frac{1}{8}, ~ 10^{-2}, ~\frac{2+\sqrt{3}}{6}$. Fitted lines have the form $a t^b$, with coefficients shown in Table~\ref{tab: 1}.} \label{fig: 4}
\end{center}
\end{figure}

The numerical results displayed in Table~\ref{tab: 1} have demonstrated the robustness of the energy stable numerical solvers in the long time scale simulations. With the time step size of order $10^{-3}$ to $10^{-2}$, and such a small interface width parameter $\varepsilon$, the numerical solver is able to capture the long time asymptotic growth rate of the surface roughness and average slope, with relative accuracy within $4\%$. Therefore, a second order accurate, energy stable numerical algorithm is a worthwhile approach for a gradient flow with a long time scale coarsening process.

\section{Concluding remarks} \label{sec:conclusion}

In this paper we have presented a linear, second order in time accurate, energy stable ETD-based scheme for a thin film model without slope selection. An unconditional long time energy stability is justified at a theoretical level. In addition, we provide an $O(\tau^2)$-order convergence analysis in the $\ell^{\infty}(0,T; \ell^2)$ norm, when $\tau \le \frac18$. Moreover, various numerical experiments have demonstrated that the proposed second-order scheme is able to produce accurate long time numerical results with a reasonable computational cost.

	\section*{Acknowledgment}
This work is supported in part by the grants NSFC 11671098, 91630309, a 111 project B08018 (W.~Chen), NSF DMS-1418689 (C.~Wang), NSF DMS-1715504 and Fudan University start-up (X.~Wang). C.~Wang also thanks the Key Laboratory of Mathematics for Nonlinear Sciences, 
Fudan University, for support during his visit.

\bibliographystyle{amsplain}

\end{document}